%% file: Suspensions_varxiv2.tex
\documentclass[12pt,centertags,oneside]{amsart}

\usepackage{amsmath,amstext,amsthm,amscd,typearea,hyperref,stmaryrd,mathabx}
\usepackage{amssymb}
\usepackage{a4wide}
\usepackage[mathscr]{eucal}
\usepackage{mathrsfs}
\usepackage{typearea}
\usepackage{charter}
\usepackage{pdfsync}
\usepackage[a4paper,width=16.2cm,top=3cm,bottom=3cm,lines=50]{geometry}
\usepackage[all]{xy}
\usepackage{tikz}
\usepackage{enumitem}
\usepackage[foot]{amsaddr}
\usepackage{academicons}
\usepackage{xcolor}

\numberwithin{equation}{section}
\newcommand{\orcid}[1]{\href{https://orcid.org/#1}{\textsc{orc}i\textsc{d}}}

\title[]{Hyperbolic foliated entropy of suspensions}

\author{Fran\c cois Bacher}
\address{Université Bourgogne Europe, CNRS, IMB UMR 5584, F-21000 Dijon, France}
\email{francois.bacher@ube.fr}

\date{\today}

\theoremstyle{plain}
\newtheorem{thm}{Theorem}[section]
\newtheorem{lem}[thm]{Lemma}

\newtheorem{prop}[thm]{Proposition}
\newtheorem{cor}[thm]{Corollary}
\newtheorem*{thm*}{Theorem}
\newtheorem*{conj*}{Conjecture}

\theoremstyle{definition}
\newtheorem{defn}[thm]{Definition}

\newtheorem*{exmp*}{Example}

\theoremstyle{remark}
\newtheorem{rem}[thm]{Remark}

\makeatletter

\@addtoreset{equation}{section}
\makeatother

\DeclareMathOperator{\id}{id}

\DeclareMathOperator{\card}{card}

\DeclareMathOperator{\Aut}{Aut}

\DeclareMathOperator{\Homeo}{Homeo}

\newcommand{\adh}[1]{\overline{#1}}

\newcommand{\PC}{P}

\newcommand{\eps}{\varepsilon}

\newcommand{\fol}{\mathscr{F}}

\newcommand{\leafatlas}{\mathscr{L}}

\newcommand{\nspllam}{\left(\mani{X},\leafatlas\right)}

\newcommand{\set}[1]{\mathbb{#1}}

\newcommand{\leaf}{L}

\newcommand{\leafu}[1]{\leaf_{#1}}

\newcommand{\intcc}[2]{\left[#1,#2\right]}

\newcommand{\intoo}[2]{\left(#1,#2\right)}

\newcommand{\intoc}[2]{\left(#1,#2\right]}

\newcommand{\intent}[2]{\left\llbracket#1,#2\right\rrbracket}

\newcommand{\sent}[1]{\lceil#1\rceil}

\newcommand{\ient}[1]{\lfloor#1\rfloor}

\newcommand{\ientp}[1]{\left\lfloor#1\right\rfloor}

\newcommand{\sentp}[1]{\left\lceil#1\right\rceil}

\newcommand{\Cmod}[1]{\left\vert#1\right\vert}

\newcommand{\class}[1]{\mathscr{C}^{#1}}

\newcommand{\smooth}{\class{\infty}}

\newcommand{\mani}[1]{#1}

\newcommand{\DR}[1]{\set{D}_{#1}}

\newcommand{\adhDR}[1]{\adh{\set{D}}_{#1}}

\newcommand{\dPC}[2]{\dPCnov(#1,#2)}

\newcommand{\dPCs}[3]{d_{\PC,#1}(#2,#3)}

\newcommand{\dPCnov}{d_{\PC}}

\newcommand{\metPC}{g_{\PC}}

\newcommand{\metm}[1]{g_{\mani{#1}}}

\newcommand{\wt}[1]{\widetilde{#1}}

\begin{document}

\theoremstyle{plain}

\begin{abstract} We study the hyperbolic entropies of foliations obtained by suspensions of a representation, in the sense of Dinh, Nguy\^{e}n and Sibony (topological and measure-theoretic). We establish a link between this type of entropy and an adapted version of an entropy defined by Ghys, Langevin and Walczak for pseudo-groups of homeomorphisms.

  Such a link has various consequences. Among them, it implies that the hyperbolic entropy of foliations is not invariant by diffeomorphisms, and that a minimal entropy suspension admits an invariant measure. Finally, this allows us to study thoroughly the simple case in which the image of the representation is isomorphic to~$\set{Z}$. In that case, we give the first exact estimate of the hyperbolic entropy, and prove a Brin--Katok type theorem and a variational principle, relying strongly on the standard ones for the entropy of maps.
  \end{abstract}

\maketitle

\section{Introduction}

The dynamical theory of laminations by Riemann surfaces has received much attention in the last twenty years. In particular, much research has been focused on the case where all the leaves are hyperbolic Riemann surfaces. In the case of foliations on projective spaces, this is the most common setup. Indeed, every polynomial vector field on~$\set{C}^n$ induces a singular holomorphic foliation on~$\set{P}^n$. For degree at least~$2$ generic foliations, all the leaves are hyperbolic. This result is due to Jouanolou~\cite{Jou}, Lins Neto and Soares~\cite{LNS}, Lins Neto~\cite{LN94} and Glutsyuk~\cite{Glu}. In that case, considering Brownian motions, one can define Lyapunov exponents of a measure. When~$n=2$, Nguy\^{e}n uses the integrability of the holonomy cocycle~\cite{Nguholo} to compute exactly the Lyapunov exponent~\cite{NguLyap} of the unique ergodic measure of a generic foliation obtained by Forn\ae{}ss and Sibony~\cite{ForSibuniergo}. It is worth noting that the unique ergodicity is extended later by Dinh, Nguy\^{e}n and Sibony~\cite{DNSergo} to a generic foliation on every compact K\"{a}hler surface. We refer the reader to the survey articles~\cite{surDinhSib,surForSib,surVANG18,surVANG21} for a more detailed exposition of this theory.

Here, we are mostly interested in an ergodic theory for laminations by hyperbolic \mbox{Riemann} surfaces. Solving heat equations with respect to harmonic currents, Dinh, Nguy\^{e}n and Sibony~\cite{DNS12} obtain an effective geometric analog of Birkhoff theorem. In this theorem, the quantity that plays the role of the time is the Poincaré distance in the uniformization of the leaves by the hyperbolic disk, and the usual invariant measures are replaced by harmonic measures. Indeed, many foliations do not admit any invariant measure. Since Garnett~\cite{Gar} (resp. Berndtsson--Sibony~\cite{BerSib}) shows that any compact lamination (resp. any compact holomorphic foliation with isolated singularities) does admit harmonic ones, we need to consider them. The hyperbolic interpretation of time then leads to notions of entropy, both measure-theoretic and topological, that the three authors define in a series of two articles~\cite{DNSI,DNSII}.

Let us introduce some notations and give a rough definition of the entropy (see Subsection~\ref{subsechypent} for the more precise one). Let $\fol=\nspllam$ be a lamination by hyperbolic Riemann surfaces, with $(X,d_X)$ a compact metric space. Then, all the leaves are uniformized by the Poincaré disk and one can choose parametrizations $\phi_x\colon\set{D}\to\leafu{x}$ such that $\phi_x(0)=x$, for $x\in X$, where~$\leafu{x}$ is the leaf passing through~$x$. For $R>0$, define a \emph{Bowen distance}
\begin{equation}\label{defdRintro}d_R(x,y)=\inf_{\theta\in\set{R}}\sup_{\xi\in\adhDR{R}}\,d_X(\phi_x(\xi),\phi_y(e^{i\theta}\xi)),\end{equation}
where $\DR{R}\subset\set{D}$ is the disk of hyperbolic radius~$R$ centered at~$0$. Note that the infimum over all possible parametrizations ensures that~$d_R$ is independent on the choice of the~$\phi_x$, for $x\in X$. Once given Bowen balls, Dinh, Nguy\^{e}n and Sibony can define a topological entropy $h(\fol)$ in the usual Bowen way. It measures the exponential rate with~$R$, at which leaves go apart with respect to~$d_R$. Also, given a harmonic measure~$\mu$, one can consider quantities introduced by Katok~\cite{Katok} and Brin--Katok~\cite{BrinKatok} as definitions, respectively for a measure-theoretic entropy $h(\mu)$, and for local upper and lower entropies $h^{\pm}(\mu,x)$ (see Section~\ref{secmesharm}).

So far, few is known about these quantities. They are finite (with no upper bound) in some cases~\cite{DNSI,DNSII,Bac3}. If~$\mu$ is ergodic and with additional asumptions~\cite{DNSI,Bac4}, the $h^{\pm}(\mu,x)$ are known to be constant almost everywhere and we have
\[2\leq h^-(\mu)\leq h(\mu)\leq h^+(\mu)\leq h(\fol).\]
In particular, there is no clear interpretation of the fact that the entropy is minimal equal to~$2$, neither is there any known non-trivial example in which it is computable. We expect, as in Brin--Katok theorem, that one could show $h^-(\mu)=h^+(\mu)$. We also hope a variational principle $h(\fol)=\sup_{\mu}h(\mu)$, but no example can argue in favor of such a belief. As a matter of fact, there is not even a written example on which it is known to be greater than~$2$.

This note intends to give first examples of explicit bounds and values for these entropies with suspensions. Consider a compact smooth Riemann surface~$\Sigma$ of genus $g\geq2$ and a representation $\rho\colon\Gamma=\pi_1(\Sigma)\to\Homeo(T)$, where~$(T,d)$ is a compact metric space. Then, $\Gamma$ acts diagonally on $\set{D}\times T$, and $X=\left(\set{D}\times T\right)/\Gamma$ is foliated by the image of horizontal disks $\set{D}\times\{t\}$, $t\in T$ (see Subsection~\ref{subsecsusp}). From the point of view of the entropies, such a lamination has the advantage of having completely explicit parametrizations~$\phi_x$, for $x\in X$. Hence, we can compare the entropy $h(\fol)$ with the one given by the holonomy~$\rho$. More precisely, given a symmetric set of generators~$\mathcal{G}$ of~$\Gamma$, Ghys--Langevin--Walczak~\cite{GLW} define a Bowen distance for $N\in\set{N}$,
\[d_N(t,s)=\sup\left\{d(\rho(g_1\dots g_n)(t),\rho(g_1\dots g_n)(s))~;~n\leq N,~g_1,\dots,g_n\in\mathcal{G}\right\},\qquad t,s\in T.\]
Here, the quantity that plays the role of the time is the word distance in~$\Gamma$ with generators~$\mathcal{G}$. As usual, a Bowen distance induces an entropy. Here, we use the same kind of distance with a weight~$\omega\colon\mathcal{G}\to\set{R}_+^*$, to define an entropy $h(\rho,\mathcal{G},\omega)$, closely related to the one of Ghys, Langevin and Walczak. Our main result is the following (see Theorem~\ref{hThrho} below for a more precise statement).

\begin{thm}\label{mainthm} Consider a suspension as above. There exists $K=K(\Gamma)>1$ such that
  \[2+h(\rho,\mathcal{G},\omega)\leq h(\fol)\leq 2+Kh(\rho,\mathcal{G},\omega).\]
\end{thm}

In particular, this has the following consequences. For the second one, we use a result of Ghys, Langevin and Walczak~\cite{GLW}.

\begin{cor}[Proposition~\ref{pasinvdiff} below] The entropy is \emph{not} a smooth invariant of laminations.
\end{cor}

\begin{cor}[Proposition~\ref{2impinvv2} below] Let~$\fol$ be a suspension with entropy~$2$. Then,~$\fol$ admits an invariant measure.
\end{cor}

Next, we consider the measure-theoretic entropies. For suspensions, Alvarez~\cite{Alvarez} shows that for a well-chosen probability measure~$m$ on~$\Gamma$, there is a one-to-one correspondence between $m$-stationary measures on~$T$, and harmonic measures on~$\fol$, for which invariant measures correspond to invariant measures. So, we obtain the same kind of comparison of the measure-theoretic entropies $h^-(\mu),h(\mu),h^+(\mu)$, with entropies of the action of~$\Gamma$ on~$T$. In particular, when the image of $\rho$ in $\Homeo(T)$ is generated by a single map~$f$, we obtain the following link with the usual topological entropy $h_{\text{top}}(f)$ and the usual measure-theoretic entropy $h_{\nu}(f)$ of~$f$.

\begin{thm}[Theorem~\ref{casZ} below]\label{casZintro} In the above setup, let~$\mu$ be an ergodic harmonic measure on~$\fol$ and~$\nu$ be the corresponding $m$-stationary measure on~$T$ given by Alvarez~\cite{Alvarez}. Then~$\mu$ and~$\nu$ are in fact invariant and there exists a~$K_0>0$ such that
  \begin{itemize}
  \item $h(\fol)=2+2K_0h_{\text{top}}(f)$,
  \item $h^-(\mu)=h(\mu)=h^+(\mu)=2+2K_0h_{\nu}(f)$.
  \end{itemize}
  In particular, we have the variational principle $h(\fol)=\sup_{\mu}h(\mu)$.
\end{thm}

Let us briefly describe the proofs of these results. The main problem is to compare on the one hand word distance (with weight) in~$\Gamma$, with generators~$\mathcal{G}$; and on the other hand the Poincaré distance to~$0$ in the disk. More precisely, the weight $\omega\colon\mathcal{G}\to\set{R}_+^*$ is defined so that for $\alpha=g_1\dots g_n\in\Gamma$, $\dPC{0}{\alpha(0)}$ is comparable to $\sum_{i=1}^n\omega(g_i)$. This result is obtained in Proposition~\ref{propcovest}. The second problem is to show that one can restrict our attention to a transversal and to the holonomy~$\rho$. This is done in Subsection~\ref{subsecT}, in which we show that one can decompose the entropy into a leafwise one equal to~$2$, and a transversal one given by the holonomy. Here, the arguments are very similar to~\cite{DNSI,DNSII,Bac4}. A crucial property that we show, is that one can get rid of the reparametrization issue in the definition~\eqref{defdRintro} of~$d_R$, by choosing canonical parametrizations, and that one still computes the same entropy. Finally, similar gymnastics of a different Bowen distance for the same entropy, give us that the transversal entropy is comparable to $h(\rho,\mathcal{G},\omega)$, as in Theorem~\ref{mainthm}.

Our more precise results for the measure-theoretic entropy come from a subadditive argument. If~$f$ is a generator of $\rho(\Gamma)$, we show that at the limit, the maximal power~$f^{n(R)}$ represented in~$\DR{R}$ is essentially $n(R)=K_0R$, for some constant~$K_0$. Therefore, all that remains to do is to adapt the proof of Brin--Katok theorem to an invertible case, that is, to powers of~$f$ between $-n(R)$ and $n(R)$, instead of from~$0$ to~$n$.

The article is organized as follows. In Section~\ref{seclam}, we present the notions of lamination, topological entropy, and suspension. In Section~\ref{sechyp}, we recall well-known facts about Fuchsian groups and prove our comparison of word distance and hyperbolic distance. In Section~\ref{seccomp}, we start by constructing the weighted entropy $h(\rho,\mathcal{G},\omega)$, then we show that we can restrict our attention to the transversal, and we prove the main theorem. Finally, in Section~\ref{secmesharm}, we recall the basics about invariant, stationary and harmonic measures and Alvarez' theorem, before proving Theorem~\ref{casZintro}.

\subsection*{Notations} Throughout this paper, we denote by $\set{D}$ the unit disk of $\set{C}$. For $R\in\set{R}_+^*$, we also denote by $\DR{R}=\{\xi\in\set{D}~;~\dPC{0}{\xi}<R\}$ the open disk of hyperbolic radius $R$ in $\set{D}$, where $\dPC{0}{\xi}=\ln\frac{1+\Cmod{\xi}}{1-\Cmod{\xi}}$ is the Poincaré distance between~$0$ and~$\xi$.

We will also consider a distance~$d=d_X$ on an ambient compact metric space~$X$, coming from a Riemannian metric on a manifold~$\mani{M}$. For $\phi,\psi\colon K\to X$ or $\tau,\sigma\colon K\to\set{D}$ (typically, $K=\adhDR{R}$, and the maps are in fact defined on whole~$\set{D}$), we denote by
\[d_K(\phi,\psi)=\sup_{\xi\in K}\,d(\phi(\xi),\psi(\xi));\qquad\dPCs{K}{\tau}{\sigma}=\sup_{\xi\in K}\,\dPC{\tau(\xi)}{\sigma(\xi)}.\]

For integers $p\leq q\in\set{Z}$, we denote by $\intent{p}{q}=\{p,p+1,\dots,q-1,q\}$ the integer interval between~$p$ and~$q$. If~$x$ is a real number, we denote by~$\ient{x}$ and~$\sent{x}$ the floor and ceiling parts of~$x$. That is, $\ient{x}$ (resp.~$\sent{x}$) is the greatest (resp. lowest) integer $p\in\set{Z}$ such that $p\leq x$ (resp. $p\geq x$).

Finally, we use $C$, $C'$, $C''$, etc. to denote positive constants which may change from a line to another.

\subsection*{Acknowledgments}  This work has been supported by the EIPHI Graduate school (contract ``ANR-17-EURE-0002''), by the Région ``Bourgogne Franche-Comté'', and by the French National Research Agency under the project DynAtrois (contract ``ANR-24-CE40-1163'').

The author would like to thank all those who asked ``naive'' questions about the hyperbolic entropy of foliations. This note intends to answer most of them in a more precise way than the usual waving hands.

\section{Laminations, suspensions and hyperbolic entropy}\label{seclam}

\subsection{Basics about laminations}\label{subseclam}
In this section, we recall some concepts in the context needed for this paper. In particular, we will not define laminations in the broadest sense. The reader may consult~\cite{surVANG21} for a more general definition.

\begin{defn}
  Let $(X,d_X)$ be a compact metric space. A \emph{lamination by Riemann surfaces} (from now on, ``lamination'') $\fol=\nspllam$ on~$X$ is the data of an atlas~$\leafatlas$ of~$X$, given by homeomorphisms
  \[\Phi_i\colon U_i\to\set{D}\times\set{T}_i,\]
  where $(U_i)_{i\in I}$ is an open covering of~$X$, $\set{T}_i$ are topological spaces (the \emph{transversals}), and the transition maps $\Phi_{ij}=\Phi_i\circ\Phi_j^{-1}$ are of the form
  \begin{equation}\label{chgtcartes}\set{D}\times\set{T}_j\to\set{D}\times\set{T}_i,\qquad(z,t)\mapsto\left(\psi_{ij}(z,t),\lambda_{ij}(t)\right),\end{equation}
  with $\psi_{ij}$ holomorphic in~$z$ and continuous in~$t$ and $\lambda_{ij}$ continuous.

  If the maps $\psi_{ij}$ and $\lambda_{ij}$ have more regularity, for example if they are $\class{r}$, Lipschitz, etc., we will say that \emph{the lamination $\fol$ is~$\class{r}$, Lipschitz}, etc.

 A coordinate chart $U\simeq\set{D}\times\set{T}$ as above is called a \emph{flow box}. The level sets $\Phi_i^{-1}\left(\set{D}\times\{t\}\right)$, for $t\in\set{T}_i$ are called \emph{plaques}. The form~\eqref{chgtcartes} of the transition maps ensures that the plaques are compatible, in the sense that two different plaques either have empty intersection, or have an open subset of~$\set{D}$ in common. A non-empty connected subset of~$X$ which contains each plaque it intersects, and which is minimal for this property, is called a \emph{leaf} of~$\fol$. The fact that $\psi_{ij}$ is holomorphic in~$z$ in~\eqref{chgtcartes} implies that a leaf~$\leaf$ of~$\fol$ inherits a Riemann surface structure.
\end{defn}

Here, we will only consider laminations which are embedded in a real manifold~$\mani{M}$. That way, using a partition of unity, one can build a smooth Riemannian metric~$\metm{M}$, which restricts to a Hermitian metric on the plaques. We also suppose that the distance~$d_X$ on~$X$ is the one coming from the distance induced by~$\metm{M}$ on~$\mani{M}$. This is actually the only reason why we need to consider embedded laminations.

We are mainly interested in laminations, the leaves of which are all hyperbolic \mbox{Riemann} surfaces. If $\fol=\nspllam$ is such a lamination, denote by~$\metPC{}$ the leafwise Poincaré metric. Two Hermitian metrics on a Riemann surface being pointwise proportional, there exists a map
\[\eta\colon X\to\set{R}_+^*,\quad\text{such that}\quad\eta^2\metPC{}=\metm{M}.\]
On the right hand side of the above equation, we have still denoted by~$\metm{M}$ its restriction to the tangent space to the leaves of~$\fol$. When~$X$ is compact, it is easy to see~\cite[p.~572]{DNSI} (and in the context of suspensions as in Subsection~\ref{subsecsusp}, the arguments can be made even simpler) that there exists a constant $c_0>1$ such that
\begin{equation}\label{bndeta}c_0^{-1}\leq\eta(x)\leq c_0,\qquad x\in X.\end{equation}

\subsection{Hyperbolic entropy}\label{subsechypent} Consider a lamination~$\fol=\nspllam$, the leaves of which are hyperbolic, where $(X,d_X)$ is a compact metric space. For each~$x\in X$, denote by~$\leafu{x}$ the leaf of~$\fol$ through~$x$ and by $\phi_x\colon\set{D}\to\leafu{x}$ a uniformization of~$\leafu{x}$ such that $\phi_x(0)=x$. Such a uniformization is unique up to precomposition by a rotation.

\begin{defn}[See~{\cite[Section~3]{DNSI}}]\label{defnN(Y,R,E)} For $R>0$, denote by~$\DR{R}=\{\xi\in\set{D}~;~\dPC{0}{\xi}<R\}$, where~$\dPCnov{}$ is the Poincaré distance on~$\set{D}$. Define the \emph{Bowen distance}
  \begin{equation}\label{defdR}d_R(x,y)=\inf_{\theta\in\set{R}}\sup_{\xi\in\adhDR{R}}\,d_X(\phi_x(\xi),\phi_y(e^{i\theta}\xi)),\qquad x,y\in X.\end{equation}

  Such a quantity measures the distance between the orbits of~$x$ and~$y$ until hyperbolic time~$R$, up to reparametrization by a rotation. Clearly, it is independent on the choice of uniformizations~$\phi_x,\phi_y$. The ball of center~$x\in X$ and radius $\eps\in\set{R}_+^*$ for the distance $d_R$ will be denoted by $B_R(x,\eps)$ and called the \emph{$R$-Bowen ball of center~$x$ and radius~$\eps$}.

  For $R,\eps>0$ and $Y\subset X$, define $N(Y,R,\eps)$ (resp. $N'(Y,R,\eps)$) the minimal $N\in\set{N}$ such that there exist $x_1,\dots,x_N\in X$ (resp. $x_1,\dots,x_N\in Y$) with
  \[Y\subset\bigcup\limits_{i=1}^NB_R(x_i,\eps).\]
  The subset $\{x_1,\dots,x_N\}$ is called \emph{$(R,\eps)$-covering of~$Y$} (resp. \emph{$(R,\eps)$-dense in~$Y$}). Also, denote by $M(Y,R,\eps)$ the maximal $M\in\set{N}$ such that there exist $x_1,\dots,x_M\in Y$ with the $B_R(x_i,\eps)$, $i\in\{1,\dots,M\}$, pairwise disjoint. The subset $\{x_1,\dots,x_M\}$ is called \emph{$(R,\eps)$-separated in~$Y$}.
\end{defn}

Such numbers are in fact closely related to each other, as can be seen in the following.

\begin{lem}[Dinh--Nguy\^{e}n--Sibony~{\cite[Proposition~3.1]{DNSI}}] \label{lienNN'M}In the above setup, we have
  \[N(Y,R,\eps)\leq N'(Y,R,\eps)\leq M(Y,R,\eps)\leq N(Y,R,\eps/2).\]
  \end{lem}

\begin{defn}[See~{\cite[Section~3]{DNSI}}] For $Y\subset X$, the \emph{hyperbolic entropy of~$Y$} is defined as the quantity
  \begin{equation}\label{defentropie}h(Y)=\sup_{\eps>0}\limsup_{R\to+\infty}\frac{1}{R}\log N(Y,R,\eps).\end{equation}

  For $Y=X$, we denote it by $h(\fol)$. By Lemma~\ref{lienNN'M}, the same equivalent definition could have been done with the numbers $N'(Y,R,\eps)$ or $M(Y,R,\eps)$.
\end{defn}

\begin{rem} The above entropy will sometimes be referred as a \emph{topological} entropy, as opposed to \emph{measure-theoretic} entropy (see Section~\ref{secmesharm}). As we will see in Proposition~\ref{pasinvdiff}, it is not necessarily invariant under homeomorphisms, so it may not be the best choice of terminology.

  On the other hand, for compact laminations, it is easy to check that it is independent on the choice of the Riemannian metric~$\metm{M}$. Also, it is invariant under applications which are transversally continuous and leafwise biholomorphic.
\end{rem}

\begin{rem} \label{gendefentropie} This setup can be made much more general by considering a metric space~$X$ and a family of distances $(d_R)_{R\geq0}$ on~$X$. In practise, we suppose that $d_0=d_X$, that $d_R$ increases with~$R$ and most of the time, the distances~$d_R$ are continuous with respect to~$d_X$. Then, there is a notion of $d_R$-balls of radius~$\eps$, of covering, dense and separated subsets. Lemma~\ref{lienNN'M} still holds and one can define an entropy for this family of distances by the formula~\eqref{defentropie}. See~\cite[Section~3]{DNSI} for various dynamically meaningful setups. This is something we will constantly do, for families of distances that will induce the same entropy (or a comparable one) as the one of~$d_R$. See for example Subsection~\ref{subsecweight} or Remark~\ref{remhh'}.
\end{rem}

\subsection{Suspensions}\label{subsecsusp}

Consider a compact smooth Riemann surface~$\Sigma$ of genus~$g\geq2$ and denote by $\Gamma=\pi_1(\Sigma)$ its fundamental group. More details about the structure of~$\Gamma$ will be recalled in the next subsection. For now, consider a representation~$\rho\colon\Gamma\to\Homeo(T)$, where~$T$ is a compact metric space. Then,~$\Gamma$ acts diagonally on $\set{D}\times T$ by
\[\Gamma\times(\set{D}\times T)\to\set{D}\times T,\qquad\gamma\cdot(z,t)=(\gamma\cdot z,\rho(\gamma)(t)),\]
where $\gamma\cdot z$ is the action by deck transformations of $\gamma\in\pi_1(\Sigma)$ on a fixed uniformization~$\set{D}\to\Sigma$. Since the action of~$\Gamma$ on~$\set{D}$ is free, properly discontinuous and cocompact, it is also such on~$\set{D}\times T$. Therefore, if we define
\[X=\left(\set{D}\times T\right)/\Gamma,\]
then~$X$ is compact. We will suppose that it is embedded in a real manifold~$M$ (this is the case if for example~$T$ is itself a real manifold), to inherit a well-chosen Riemannian structure. For $A\subset\set{D}\times T$, denote by $[A]$ its image under the quotient map $\set{D}\times T\to X$. Then, $X$ is laminated by the hyperbolic Riemann surfaces $[\set{D}\times\{t\}]$, for $t\in T$. All these leaves are coverings of~$\Sigma$. Note $\fol=\nspllam{}$ this lamination. It is called the \emph{suspension of the representation~$\rho$}.

Above, we presented suspensions in the context that is necessary to define our entropy, that is, the one of hyperbolic leaves. The reader can consult~\cite[Chapter~V, \S4]{CLNgeomfol} for a broader definition. In our setup, Ghys~\cite[pp.~51--52]{Ghys} (see also~\cite[Example~2.31]{surVANG21}) shows that when~$T=\set{P}^1$ and $\rho(\Gamma)\subset\Aut(\set{P}^1)$, then~$X$ is an algebraic surface that can be embedded in~$\set{P}^N$, for some $N\geq3$.

From the point of view of the entropy, suspensions constructed as above have the important and simple property that the covering maps of the leaves are completely explicit. This will allow us to have good estimates of the entropy, in terms of the representation~$\rho$.

More precisely, denote by $\set{T}_0=[\{0\}\times T]$. This is a global transversal to the lamination~$\fol$, which is isomorphic to~$T$, and that will be called the \emph{distinguished transversal}. We will sometimes make the confusion between a point $t\in T$ and the point $[0,t]\in\set{T}_0$. For $t\in T$, denote by $\phi_t$ the uniformization of the leaf $\leafu{t}$ passing through~$t$ given by
\begin{equation}\label{defphit}\phi_t\colon\set{D}\to\leafu{t};\qquad z\mapsto[z,t].\end{equation}
It is of course such that $\phi_t(0)=t$ (with the abuse $t\simeq[0,t]$), as in the notations of Subsection~\ref{subsechypent}.

\section{Some hyperbolic geometry}\label{sechyp}

For this section, our main references are Katok's~\cite{KatokFuchs} and Farb--Margalit's~\cite{Farb} books. We will only focus on the amount of hyperbolic geometry needed for further estimates of the entropy. Also, on the Fuchsian group's side, we will only consider those which are fundamental groups of smooth compact Riemann surfaces. The reader is invited to consult the aforementioned books for a more complete presentation.

\subsection{Cocompact Fuchsian groups} \label{subsecFuchs}

In this paragraph, we recall what is a Fuchsian group, with a focus on those which are fundamental groups of compact smooth hyperbolic Riemann surfaces.

\begin{defn} A \emph{Fuchsian group} is a subgroup~$\Gamma$ of $\Aut(\set{D})$ that acts properly and discontinuously on~$\set{D}$, that is, for every compact~$K$ of~$\set{D}$, $\{\alpha\in\Gamma~;~\alpha(K)\cap K\neq\emptyset\}$ is finite.

  If~$\Gamma$ is a Fuchsian group, the quotient $\Sigma=\set{D}/\Gamma$ is a hyperbolic Riemann surface (with possibly orbifold singularities). We say that~$\Gamma$ is \emph{cocompact} if~$\Sigma$ is compact.
\end{defn}

  Moreover, $\Sigma$ is smooth if and only if every element of~$\Gamma$ besides the identity has no fixed point in~$\set{D}$ and in that case, $\pi_1(\Sigma)=\Gamma$ (see~\cite[\S3.6]{KatokFuchs}). Conversely, let~$\Sigma$ be a compact smooth hyperbolic Riemann surface and let $\phi\colon\set{D}\to\Sigma$ be a uniformization of~$\Sigma$. Then, one can write $\Sigma=\set{D}/\Gamma$, with~$\Gamma=\pi_1(\Sigma)$ being a Fuchsian group acting by deck transformations on~$\set{D}$, and every element of~$\Gamma$ besides the identity having no fixed point.

  The next results collect some very well-known facts we will use about Fuchsian groups.

  \begin{prop}[{\cite[Corollary~4.2.3]{KatokFuchs}}] With the notations above, $\Sigma$ is compact if and only if~$\Gamma$ has a compact fundamental domain~$D$.
  \end{prop}

  \begin{prop}[{\cite[10.4.1]{Farb}}]\label{presFuchs} Consider two smooth compact Riemann surfaces $\Sigma_i=\set{D}/\Gamma_i$ of genus $g_i\geq2$, for $i\in\{1,2\}$. Then, $\Sigma_1$ and $\Sigma_2$ are homeomorphic if and only if they are diffeomorphic, if and only if $\Gamma_1$ and $\Gamma_2$ are group-isomorphic, if and only if $g=g_1=g_2$. In that case, the $\Gamma_i$ have the presentation
    \[\Gamma_i=\left\langle\alpha_1,\dots,\alpha_g,\beta_1,\dots,\beta_g~\vert~\alpha_1\beta_1\alpha_1^{-1}\beta_1^{-1}\dots\alpha_g\beta_g\alpha_g^{-1}\beta_g^{-1}=\id\right\rangle.\]
  \end{prop}

  Now, we prove a simple result that will be useful in Subsection~\ref{subseccsq} to provide examples of diffeomorphic suspensions having different entropies.
  
\begin{prop}\label{existdegFuchs} Let $g\in\set{N}$ with $g\geq2$ and $\eps>0$. Then, there exists a Fuchsian group~$\Gamma_{\eps}$, with $\set{D}/\Gamma_{\eps}$ a compact smooth Riemann surface of genus~$g$, a set of generators $\mathcal{G}_{\eps}$ of~$\Gamma_{\eps}$ with $\mathcal{G}_{\eps}=\mathcal{G}_{\eps}^{-1}$, $\id\notin\mathcal{G}_{\eps}$, and an element $\alpha_{\eps}\in\mathcal{G}_{\eps}$ such that $\dPC{0}{\alpha_{\eps}(0)}\leq\eps$.
\end{prop}

\begin{proof} The proof goes as for Poincaré's theorem~\cite[Theorem~3.4.2]{KatokFuchs}. Consider a convex $4g$-gone, bounded by geodesics, the vertex of which are $re^{i\theta_j}$, $j\in\intent{1}{4g}$, with $\theta_1=\mu/2$, $\theta_2=\mu/2+\nu$, $\theta_{j+1}=\theta_j+\mu$, $j\in\intent{2}{4g-2}$, $\theta_{4g}=\theta_{4g-1}+\nu=2\pi-\mu/2$ (see Figure~\ref{figdegFuchs}). In other words, we choose one big angle~$\nu$ and one small angle~$\mu$, with $2\nu+(4g-2)\mu=2\pi$, put $4g$ points on a circle with two pairs of them at angle~$\nu$ and the others at angle~$\mu$, so that the two big geodesics are separated by only one geodesic. Denote by $A_j=re^{i\theta_j}$ and by~$D$ the closed~$4g$-gone bounded by the geodesics $A_jA_{j+1}$, $j\in\intent{1}{4g}$, with $A_{4g+1}=A_1$. Now, as for Poincaré theorem, the sum of the internal angles of~$D$ goes to $(4g-2)\pi$ when~$r$ goes to~$0$, and to~$0$ when~$r$ goes to~$1$. So, there is an~$r\in\intoo{0}{1}$ such that it is~$2\pi$, and the usual identification of the sides of~$D$ defines a Fuchsian group (see~\cite[10.4.2]{Farb}).

  We wish to show that it satisfies our condition. Denote by~$\alpha_0$ the generator sending the geodesic~$(A_{4g-1}A_{4g})$ to the geodesic~$(A_1A_2)$. It is enough to prove that for well chosen~$\mu$ and~$\nu$, $\dPC{0}{\alpha_0(0)}<\eps$. First, let us estimate this distance in terms of these angles. Note~$B_1$ the middle of the geodesic $[A_{4g-1}A_{4g}]$ and $B_2$ the middle of $[A_1A_2]$. Then $\alpha_0(B_1)=B_2$ and since $B_1$ and~$B_2$ are complex conjugate,
  \begin{equation}\label{dPC0alpha0}\dPC{0}{\alpha_0(0)}\leq\dPC{0}{B_2}+\dPC{\alpha_0(B_1)}{\alpha_0(0)}=2\dPC{0}{B_2}.\end{equation}
  In fact, we even have the equality but we do not need it. Now, $B_2$ is a point $r'e^{i\left(\mu+\nu\right)/2}$ and is in the Euclidean triangle $0A_{1}A_{2}$, which is itself contained in the Euclidean triangle, the vertex of which are $0$, $e^{i\mu/2}$, $e^{i(\mu/2+\nu)}$. By elementary trigonometry, $r'\leq\cos(\nu/2)$ and if~$\nu$ is sufficiently close to~$\pi$, this implies that $\dPC{0}{B_2}<\frac{\eps}{2}$. Equation~\eqref{dPC0alpha0} then gives the result and concludes the proof. 
\end{proof}

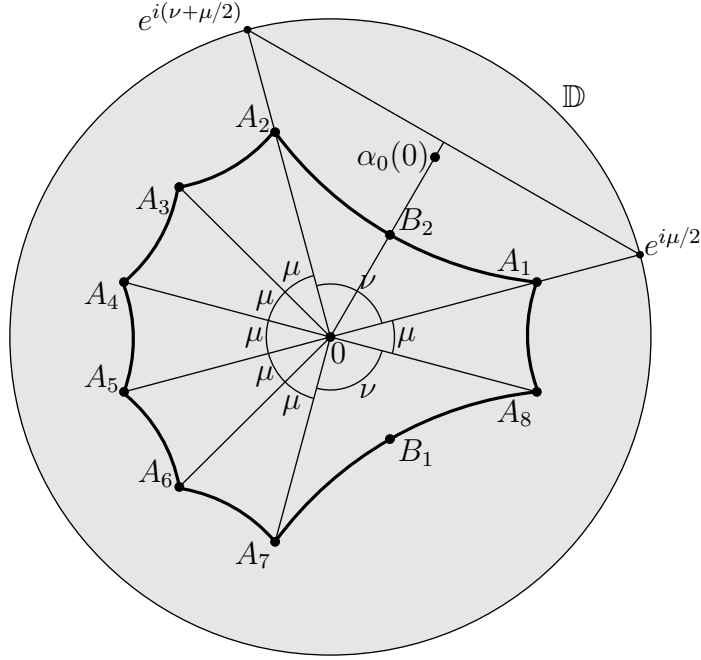
\begin{figure}[tbh]
  \centering
  \input{Deg-Fuchs.tex}
  \caption{\label{figdegFuchs} The polygon~$D$, for some~$r\in\intoo{0}{1}$ and~$g=2$, in the proof of Proposition~\ref{existdegFuchs}.}
  \end{figure}







\subsection{Covering estimates}\label{subseccover}

In this subsection, we give estimates that allow to compare hyperbolic distance with word distance in a Fuchsian group with generators. Our main result is the following.

\begin{prop}\label{propcovest} Let $\Gamma$ be a Fuchsian group with $\set{D}/\Gamma$ being a smooth compact Riemann surface of genus $g\geq2$. Denote by $D$ a compact fundamental domain of~$\Gamma$, with $0\in D$, and by $\mathcal{G}$ a set of generators such that $\mathcal{G}=\mathcal{G}^{-1}$. Define $R_0=\max\{\dPC{0}{\alpha(0)}~;~\alpha\in\mathcal{G}\}$. Then, there exists $K=K(\Gamma)>1$, such that for any $\delta\in\intoo{0}{K^{-1}}$ and for $N\in\set{N}$ sufficiently large,
  \[\set{D}_{(K^{-1}-\delta)N}\subset\bigcup\limits_{\underset{n\leq N}{\alpha_1,\dots,\alpha_n\in\mathcal{G}}}\alpha_1\dots \alpha_n(D)\subset\set{D}_{N(R_0+\delta)}.\]
\end{prop}

\begin{proof}[Beginning of the proof] Let us start with the right inclusion. Let $\delta_0$ be the diameter of $D$ and $N\in\set{N}$ be sufficiently large to have $N\delta>\delta_0$. Take $\alpha_1,\dots,\alpha_n\in\mathcal{G}$ for $n\leq N$, $\zeta\in D$, and denote by $\xi=\alpha_1\dots\alpha_n(\zeta)$. We need to show that $\dPC{0}{\xi}<N(R_0+\delta)$. Denote by $\xi_k=\alpha_1\dots\alpha_k(0)$, for $k\in\intent{0}{n}$, with $\xi_0=0$. Using triangle inequality,
  \[\dPC{0}{\xi}\leq\sum_{k=0}^{n-1}\dPC{\xi_k}{\xi_{k+1}}+\dPC{\xi_n}{\xi}=\sum_{k=0}^{n-1}\dPC{0}{\alpha_{k+1}(0)}+\dPC{0}{\zeta}.\]
  Above, we used that all the $\alpha_j$, $j\in\intent{1}{n}$, are isometries of the Poincar\'e distance. Now, by definition, all the terms of the sum are at most~$R_0$ and $\dPC{0}{\zeta}\leq\delta_0$. Hence,
  \[\dPC{0}{\xi}\leq nR_0+\delta_0<N(R_0+\delta),\]
  thanks to the above choice of~$N$. This concludes the proof of the right inclusion.
\end{proof}

\begin{rem} In Section~\ref{seccomp} (see Proposition~\ref{pasinvdiff}), we will encounter situations in which some generators~$\alpha\in\mathcal{G}$, on the one hand have very small distance $\dPC{0}{\alpha(0)}$, thanks to Proposition~\ref{existdegFuchs}; and on the other hand are ``the only ones to count''. Indeed, in the representation~$\rho$ giving a suspension, the other generators will be sent to the identity, so that we will be able not to mind them. In such a situation, we will need a refinement of the inclusion we just proved. Namely, consider $\alpha_1,\dots,\alpha_n\in\mathcal{G}$, $\omega_j=\dPC{0}{\alpha_j(0)}$, for $j\in\intent{1}{n}$. Then, with the same proof as above,
  \begin{equation}\label{rempropcovset}\alpha_1\dots \alpha_n(D)\subset\set{D}_{\sum_{j=1}^n\omega_j+\delta_0}.\end{equation}
\end{rem}

For the other inclusion, and for finding~$K(\Gamma)$, we need one lemma.

\begin{lem}\label{lemcovest} Recall that $\delta_0$ was defined as the diameter of~$D$, with the notations of Proposition~\ref{propcovest}. There exists a finite number $\beta_1,\dots,\beta_p$ of elements of~$\Gamma$ such that for any~$\xi\in\set{D}$ with $\dPC{0}{\xi}>\delta_0+1$,
  \[\dPC{\beta_k(0)}{\xi}\leq\dPC{0}{\xi}-1,\qquad\text{for some}~k\in\intent{1}{p}.\]
\end{lem}

\begin{proof} Let us first construct some~$\beta_k$, $k\in\intent{1}{p}$, that satisfy
  \begin{equation}\label{covcercle}\partial\DR{\delta_0+1}\subset\bigcup\limits_{k=1}^p\beta_k(D).\end{equation}
  Since~$\Gamma$ acts properly discontinuously on~$\set{D}$, the set $B=\{\beta\in\Gamma~;~\dPC{0}{\beta(0)}\leq 2\delta_0+1\}$  is finite. Now, for $\xi\in\partial\DR{\delta_0+1}$, there exists $\beta\in\Gamma$ such that $\xi\in\beta(D)$. But,
  \[\dPC{0}{\beta(0)}\leq\dPC{0}{\xi}+\dPC{\xi}{\beta(0)}\leq 2\delta_0+1,\]
  since $\xi=\beta(\zeta)$, $\zeta\in D$,~$\beta$ is an isometry, and by definition of~$\delta_0$. Hence, $\beta\in B$, and this completes the proof of~\eqref{covcercle}, denoting by $B=\{\beta_1,\dots,\beta_p\}$.

  With the $\beta_1,\dots,\beta_p$ constructed, let us turn to the proof of the lemma. Let $\xi\in\set{D}$ with $\dPC{0}{\xi}>\delta_0+1$ and let $\xi_1$ be the point in the geodesic joining~$0$ to~$\xi$, lying between~$0$ and~$\xi$ and belonging to $\partial\DR{\delta_0+1}$. By~\eqref{covcercle}, there exist $k\in\intent{1}{p}$ and $\zeta\in D$ such that $\xi_1=\beta_k(\zeta)$. Now, by triangle inequality,
  \[\dPC{\beta_k(0)}{\xi}\leq\dPC{\beta_k(0)}{\beta_k(\zeta)}+\dPC{\xi_1}{\xi}\leq\delta_0+\dPC{0}{\xi}-(\delta_0+1)=\dPC{0}{\xi}-1.\]
  Above, we used that~$\beta_k$ is an isometry of the Poincaré distance, that $\dPC{0}{\zeta}\leq\delta_0$ by definition of~$\delta_0$, and that $\dPC{0}{\xi}=\dPC{0}{\xi_1}+\dPC{\xi_1}{\xi}$, thanks to their configuration on the same geodesic.
\end{proof}

\begin{proof}[Proposition~\ref{propcovest}'s end of proof] Consider the $\beta_k$, $k\in\intent{1}{p}$, given by Lemma~\ref{lemcovest} and decompose each of them $\beta_k=\alpha_{1,k}\dots\alpha_{\ell_k,k}$ as a product of generators $\alpha_{j,k}\in\mathcal{G}$, for $k\in\intent{1}{p}$ and $j\in\intent{1}{\ell_k}$. Denote by $K=\max\{\ell_k~;~k\in\intent{1}{p}\}$.

  Now, arguing as for~\eqref{covcercle}, note that $\adhDR{\delta_0+1}$ is covered by a finite number of copies of the fundamental domain $D$. Hence, there is $N_0\in\set{N}$ such that for $\xi\in\adhDR{\delta_0+1}$, one can find $\alpha_1,\dots,\alpha_n\in\mathcal{G}$, with $n\leq N_0$ and $\xi\in\alpha_1\dots\alpha_n(D)$. Let us show by induction on~$L\in\set{N}$, that if $\xi\in\set{D}$ satisfies $\dPC{0}{\xi}\leq\delta_0+1+L$, then there are $\alpha_1,\dots,\alpha_n\in\mathcal{G}$, with $n\leq KL+N_0$ and $\xi\in\alpha_1\dots\alpha_n(D)$. We have just done the initialization for $L=0$. Let~$L\in\set{N}$, suppose the induction hypothesis satisfied up to rank~$L$ and let $\xi\in\set{D}$ be such that $\dPC{0}{\xi}\in\intoc{\delta_0+L+1}{\delta_0+L+2}$. By Lemma~\ref{lemcovest}, one finds~$\beta_k$, for some $k\in\intent{1}{p}$, with $\dPC{\beta_k(0)}{\xi}\leq\dPC{0}{\xi}-1$. Denote by $\xi'=\beta_k^{-1}(\xi)$. Then, $\dPC{0}{\xi'}\leq\delta_0+L+1$. By induction hypothesis, there exist $\alpha_1,\dots,\alpha_n\in\mathcal{G}$, with $n\leq KL+N_0$ and $\xi'\in\alpha_1\dots\alpha_n(D)$. Hence,
  \[\xi=\beta_k(\xi')\in\beta_k\alpha_1\dots\alpha_n(D)=\alpha_{1,k}\dots\alpha_{\ell_k,k}\alpha_1\dots\alpha_n(D),\]
  with $n+\ell_k\leq K(L+1)+N_0$, since $\ell_k\leq K$ by definition of~$K$. The induction is complete.

  Finally, consider $\delta\in\intoo{0}{K^{-1}}$ and~$N\in\set{N}$ sufficiently large. Let~$\xi\in\DR{(K^{-1}-\delta)N}$, and define $L=\sentp{N(K^{-1}-\delta)-(\delta_0+1)}$. Recall that $\sent{x}$ denotes the smallest integer~$n\in\set{Z}$ such that $n\geq x$, for $x\in\set{R}$. Applying the result of the previous paragraph, we obtain $\alpha_1,\dots,\alpha_n\in\mathcal{G}$, $n\leq KL+N_0$, such that $\xi\in\alpha_1\dots\alpha_n(D)$. So, it is enough to show that if~$N$ is sufficiently large, $KL+N_0\leq N$. By minimality of~$L$, we have
  \[KL+N_0\leq K\left(N(K^{-1}-\delta)-\delta_0\right)+N_0\leq N-KN\delta+N_0\leq N,\]
  if $N\geq N_0K^{-1}\delta^{-1}$. This completes the proof.
\end{proof}


\section{Comparing entropies}\label{seccomp}
  
\subsection{Entropy of a homeomorphisms representation with weigthed generators}\label{subsecweight}

Here, we consider the abstract setup of a compact metric space~$(T,d_T)$, a group~$\mathcal{H}$, with a (typically finite) set of generators~$\mathcal{G}$ such that $\mathcal{G}=\mathcal{G}^{-1}$, and a morphism $\rho\colon\mathcal{H}\to\Homeo(T)$. We will construct some quantity that generalizes slightly an entropy defined by Ghys, Langevin and \mbox{Walczak}~\cite[Section~2]{GLW} in this setup.

Consider also $\omega\colon\mathcal{G}\to\set{R}_+^*$ a \emph{weight function}. Define the \emph{Bowen distance} for $R>0$ and $t,s\in T$ to be
\[d_R^{\rho,\mathcal{G},\omega}(t,s)=\sup\left\{d_T(\rho(g_1\dots g_k)(t),\rho(g_1\dots g_k)(s))\right\},\]
where the $\sup$ is taken over $g_i\in\mathcal{G}$, $i\in\intent{1}{k}$ such that $\sum_{i=1}^k\omega(g_i)\leq R$, for any $k\in\set{N}$.

Once given such a Bowen distance, we can mimick the vocabulary of Subsection~\ref{subsechypent} (see Remark~\ref{gendefentropie}) about $(R,\eps)$-covering, dense and separated subsets for $(\rho,\mathcal{G},\omega)$. Denote by $B_R^{\rho,\mathcal{G},\omega}(t,\eps)$ the Bowen ball $\left\{s\in T~;~d_R^{\rho,\mathcal{G},\omega}(t,s)<\eps\right\}$, for $t\in X$ and $R,\eps>0$. Let $R,\eps>0$ and $Y\subset T$. Define $N(Y,R,\eps,\rho,\mathcal{G},\omega)$ (resp. $N'(Y,R,\eps,\rho,\mathcal{G},\omega)$) to be the minimal cardinality of an $(R,\eps)$-covering subset of~$Y$ (resp. dense subset in~$Y$), and $M(Y,R,\eps,\rho,\mathcal{G},\omega)$ the maximal cardinality of an $(R,\eps)$-separated subset in~$Y$. With exactly the same proof, we have the same property as Lemma~\ref{lienNN'M}. Hence, define the \emph{entropy of~$\rho$ with respect to the set of generators~$\mathcal{G}$ and weight~$\omega$} by
\[h(Y,\rho,\mathcal{G},\omega)=\sup\limits_{\eps>0}\limsup\limits_{R\to +\infty}\frac{1}{R}\log N(Y,R,\eps,\rho,\mathcal{G},\omega).\]
Of course, such a definition yields the same result if we had taken $N'(Y,R,\eps,\rho,\mathcal{G},\omega)$ or $M(Y,R,\eps,\rho,\mathcal{G},\omega)$ instead. If $Y=T$, denote it simply by $h(\rho,\mathcal{G},\omega)$.

\begin{rem}\label{remlienhGLWh} Let us make a comparison with the similar definition of Ghys, Langevin and \mbox{Walczak}~\cite[Section~2]{GLW}. They define the same kind of entropy (without making the Bowen distance explicit) for pseudo-groups of homeomorphisms. Here, we don't need pseudo-groups, since the holonomies of suspensions are global. However, we prefer defining it at the level of representations and not at the level of groups, since we will consider contexts in which two generators have the same image but not the same weight.

  Let us denote by $h_{\rm{GLW}}(\rho(\mathcal{H}),\rho(\mathcal{G}))$ the entropy defined by Ghys--Langevin--Walczak. For clarity reasons, they need to include the identity in~$\mathcal{G}$ but we ignore this issue (or we add it to our setup with $\omega(\id)=0$). Our entropy generalizes theirs in the sense that if~$1$ denotes the constant weight equal to~$1$, then $h_{\rm{GLW}}(\rho(\mathcal{H}),\rho(\mathcal{G}))=h(\rho,\mathcal{G},1)$. More generally, we have the following.
\end{rem}

\begin{lem}\label{lemlienhGLWh}
  Keep the notations above. If~$c_1\leq\omega\leq c_2$, for $c_1,c_2\in\set{R}_+^*$ (this always happens for some $c_1,c_2\in\set{R}_+^*$ if~$\mathcal{G}$ is finite), then
  \[c_2^{-1}h_{\rm{GLW}}(\rho(\mathcal{H}),\rho(\mathcal{G}))\leq h(\rho,\mathcal{G},\omega)\leq c_1^{-1}h_{\rm{GLW}}(\rho(\mathcal{H}),\rho(\mathcal{G})).\]
\end{lem}

\begin{proof}It follows from the definition of $d_R^{\rho,\mathcal{G},\omega}$ that $d_{c_2^{-1}R}^{\rho,\mathcal{G},1}\leq d_R^{\rho,\mathcal{G},\omega}\leq d_{c_1^{-1}R}^{\rho,\mathcal{G},1}$. Then, separated (resp. covering) sets for $d_R^{\rho,\mathcal{G},\omega}$ being separated (resp. covering) for $d_{c_1^{-1}R}^{\rho,\mathcal{G},\omega}$ (resp. $d_{c_2^{-1}R}^{\rho,\mathcal{G},\omega}$), we obtain
  \[N(Y,c_2^{-1}R,\eps,\rho,\mathcal{G},1)\leq N(Y,R,\eps,\rho,\mathcal{G},\omega),\quad M(Y,R,\eps,\rho,\mathcal{G},\omega)\leq M(Y,c_1^{-1}R,\eps,\rho,\mathcal{G},1).\]
  Applying the $\log$, dividing by~$R$ and taking the relevant limits, it yields
  \[c_2^{-1}h(Y,\rho,\mathcal{G},1)\leq h(Y,\rho,\mathcal{G},\omega)\leq c_1^{-1}h(Y,\rho,\mathcal{G},1).\]
  We finish the proof by applying Remark~\ref{remlienhGLWh}.
\end{proof}

Once is proven that our entropy is linked to the three authors', we refer the reader to their article~\cite{GLW} for examples, estimates, and links to other notions of entropies. 

\subsection{Entropy of the distinguished transversal}\label{subsecT}

From now on, fix $\fol=\nspllam{}$ a lamination given by the suspension process of Subsection~\ref{subsecsusp}. Keep the same notations as those introduced in it. Also, denote by~$D$ a compact fundamental domain of~$\Gamma$, with $0\in D$, as was done in Subsection~\ref{subseccover} and denote by~$\delta_0$ the diameter of~$D$. Moreover, consider a smooth Riemannian metric~$\metm{M}$ on~$X$, which restricts to a Hermitian metric on the leaves, as explained in Subsection~\ref{subseclam}. Denote by~$d$ the distance on~$X$ induced by~$\metm{M}$. We want to reduce the problem of computing the entropy to the distinguished transversal~$\set{T}_0$. More precisely, we first want to obtain the following result.

\begin{thm}\label{hFhT} With the notations above, $h(\fol)=h(\set{T}_0)+2$.
\end{thm}

To prove the latter, we need several lemmas. The first two concern how fast two automorphisms of the disk diverge from each other if they are close in~$0$. The first one comes from our previous work. The formulation is slightly modified for our purposes, but the statement is exactly the one that is proven (with the notations below, because $e^{-R}\leq1-\Cmod{\xi}^2\leq2e^{-R}$, if~$R$ is sufficiently large). The second one comes from Dinh, Nguy\^{e}n and Sibony's work on the entropy of foliations. A more precise version of it can be found in~\cite[Lemma~3.14]{Bac4}, but the one here is sufficient for our needs.

\begin{lem}[{\cite[Lemma~3.11]{Bac4}}]\label{lemBac4angles} Let~$\eps>0$ be sufficiently small, $\theta\in\intcc{-\pi}{\pi}$ and $\xi\in\set{D}$ be such that $R=\dPC{0}{\xi}$ is sufficiently large.
  \begin{enumerate}[label=(\roman*),ref=\roman*]
  \item \label{lemBac4<} If $e^{-R}\geq 8\eps^{-1}\Cmod{\sin(\theta/2)}$, then $\dPC{\xi}{e^{i\theta}\xi}\leq\eps$.
  \item \label{lemBac4>} If $e^{-R}\leq \frac{1}{4}\eps^{-1}\Cmod{\sin(\theta/2)}$, then $\dPC{\xi}{e^{i\theta}\xi}\geq\eps$.
  \end{enumerate}
\end{lem}

\begin{lem}[{\cite[Lemma~3.8]{DNSI}}]\label{lemDNS1Aut}
  Let $\eps>0$ be sufficiently small. There exists $A>1$ such that for all $a,b\in\set{D}$, we have the following.
  \begin{enumerate}[label=(\roman*),ref=\roman*]
  \item\label{lemDNSI<} If $\dPC{a}{b}\leq A^{-1}e^{-R}$, then there exist $\tau_a,\tau_b\in\Aut(\set{D})$ such that $\tau_a(0)=a$, $\tau_b(0)=b$ and $\dPCs{\adhDR{R}}{\tau_a}{\tau_b}\leq\eps$.
  \item\label{lemDNSI>} If $\dPC{a}{b}\geq Ae^{-R}$, then for all $\tau_a,\tau_b\in\Aut(\set{D})$ such that $\tau_a(0)=a$ and $\tau_b(0)=b$, $\dPCs{\adhDR{R}}{\tau_a}{\tau_b}\geq\eps$.
  \end{enumerate}
\end{lem}
  Above, recall the notation
  \[\dPCs{K}{\tau}{\sigma}=\sup_{\xi\in K}\,\dPC{\tau(\xi)}{\sigma(\xi)},\]
  for $K\subset\set{D}$ and $\tau,\sigma\in\Aut(\set{D})$.

  Now, we also need a simple result which compares the Poincaré metric on the disk and on close points in~$S$. Recall the notation~$\phi_t$, for $t\in\set{T}_0$, introduced in~\eqref{defphit}.
  
  \begin{lem}\label{lemh>dp>eps} There exist $h>0$ and $c>1$ such that
    \begin{enumerate}[label=(\roman*),ref=\roman*]
    \item \label{dp<eps} if $t\in\set{T}_0$ and $\zeta,\xi\in\set{D}$, then $d(\phi_t(\zeta),\phi_t(\xi))\leq c\dPC{\zeta}{\xi}$;
    \item \label{dp>eps} if $t,s\in\set{T}_0$, $\zeta,\xi\in\set{D}$ are such that $\dPC{\zeta}{\xi}\leq h$, then $d(\phi_t(\zeta),\phi_s(\xi))\geq c^{-1}\dPC{\zeta}{\xi}$.
    \end{enumerate}
  \end{lem}

  \begin{proof} Note that $\phi_t(\zeta)=\phi_{t'}(\zeta')$, for some $\zeta'\in D$ and $t'\in\set{T}_0$. Then, the statement~\eqref{dp<eps} follows immediatly from the equivalence of Hermitian metrics~\eqref{bndeta} and a compactness argument. For statement~\eqref{dp>eps}, let us choose $h>0$ such that
    \begin{equation}\label{condh}\dPC{\xi}{\alpha(\xi)}>2h,\qquad\xi\in\set{D},\alpha\in\Gamma\setminus\{\id\}.\end{equation}
    It is classical that such an~$h$ exists, since~$\Gamma$ is cocompact and since each $\alpha\in\Gamma\setminus\{\id\}$ has no fixed point. Reducing~$h$ if necessary, it is sufficient to prove the statement only for $\phi_t(\zeta)$ and $\phi_s(\xi)$ belonging to a common chart $U\simeq\set{D}\times\set{B}^p$ of the ambient Riemannian manifold~$\mani{M}$, where $\set{B}^p$ is the unit ball of some~$\set{R}^p$. Take a Riemannian geodesic joining $\phi_t(\zeta)$ and $\phi_s(\xi)$. If~$h$ is sufficiently small, it stays inside~$U$ and its projection in the leafwise direction joins~$\zeta$ to some~$\alpha(\xi)$, $\alpha\in\Gamma$. So that it is of Poincaré length at least~$\dPC{\zeta}{\xi}$, thanks to condition~\eqref{condh}. We conclude by compactness of $D\times\set{T}_0$, the equivalence of Riemannian metrics and~\eqref{bndeta}.
  \end{proof}

  The following result will be the last of our preparation for Theorem~\ref{hFhT}, but will also prove itself very useful in the next subsection (see Remark~\ref{remhh'}).

  \begin{lem}\label{lemhh'} Let $h>0$ be given by Lemma~\ref{lemh>dp>eps} and $\delta>0$. For all $\eps>0$ sufficiently small, and all $R>0$ sufficiently large, we have the following. Let $\zeta\in\DR{h}$ and $t,s\in\set{T}_0$ be such that $d_R(\phi_t(\zeta),\phi_s(\zeta))<\eps$. Then,
    \[d_{\adhDR{(1-\delta)R}}(\phi_t,\phi_s)<2\eps.\]
  \end{lem}
  Here again and in the proof below, we use the notation
  \[d_{K}(\phi,\psi)=\sup_{\xi\in K}\,d(\phi(\xi),\psi(\xi)),\]
  for $K\subset\set{D}$ and $\phi,\psi\colon\set{D}\to X$.

  \begin{proof} Let $\tau\in\Aut(\set{D})$ be such that $\tau(0)=\zeta$ and define $\psi_t=\phi_t\circ\tau$, $\psi_s=\phi_s\circ\tau$. Since $d_R(\psi_t(0),\psi_s(0))<\eps$, one finds $\theta\in\intcc{-\pi}{\pi}$ with $d_{\adhDR{R}}(\psi_t,\psi_s\circ r_{\theta})<\eps$, where $r_{\theta}$ is the rotation of angle~$\theta$ in~$\set{D}$. We first show that $\Cmod{\theta}\leq e^{-R}$. Note that if there is some~$\xi\in\adhDR{R}$ with $\dPC{\xi}{e^{i\theta}\xi}\geq c\eps$, then by continuity, there is a~$\xi\in\adhDR{R}$ such that $h\geq\dPC{\xi}{e^{i\theta}\xi}\geq c\eps$. Applying statement~\eqref{dp>eps} of Lemma~\ref{lemh>dp>eps} to $\tau(\xi)$ and $\tau(e^{i\theta}\xi)$, it follows that $d(\psi_t(\xi),\psi_s(e^{i\theta}\xi))\geq\eps$, a contradiction. Hence, for all~$\xi\in\adhDR{R}$, $\dPC{\xi}{e^{i\theta}\xi}\leq c\eps$. By statement~\eqref{lemBac4>} of Lemma~\ref{lemBac4angles}, one obtains $\Cmod{\sin(\theta/2)}\leq 4\eps e^{-R}$, so that $\Cmod{\theta}\leq e^{-R}$, for~$\eps$ sufficiently small.

    Now, since $\adhDR{(1-\delta)R}\subset\tau\left(\adhDR{(1-\delta/2)R}\right)$, for $R$ sufficiently large,
    \[d_{\adhDR{(1-\delta)R}}(\phi_t,\phi_s)\leq d_{\adhDR{(1-\delta/2)R}}(\psi_t,\psi_s)\leq d_{\adhDR{R}}(\psi_t,\psi_s\circ r_{\theta})+d_{\adhDR{(1-\delta/2)R}}(\psi_s,\psi_s\circ r_{\theta}).\]
    Above, the first term of the right hand side is at most~$\eps$, and the second also by statement~\eqref{lemBac4<} of Lemma~\ref{lemBac4angles}. This concludes the proof.
  \end{proof}

  \begin{proof}[Proof of Theorem~\ref{hFhT}] We prove separately the two inequalities. Let us first show that $h(\fol)\leq2+h(\set{T}_0)$. Take~$F$ an $\left(R+\delta_0,\frac{\eps}{2}\right)$-dense subset in~$\set{T}_0$ of minimal cardinality and~$G$ an $A^{-1}e^{-R}$-dense subset of~$\DR{\delta_0}$ for the Poincaré distance of minimal cardinality. Here, $A$ is the constant of Lemma~\ref{lemDNS1Aut}. Define $H=\{\phi_t(\zeta)~;~t\in F,~\zeta\in G\}$. We claim that~$H$ is $(R,\eps)$-dense in~$X$. Let $x=\phi_s(\xi)\in X$, for $s\in\set{T}_0$ and $\xi\in D$. Take $t\in F$ with $d_{R+\delta_0}(t,s)<\frac{\eps}{2}$. In particular, one finds $\theta\in\set{R}$ with $d_{\adhDR{R+\delta_0}}(\phi_t,\phi_s\circ r_{\theta})<\frac{\eps}{2}$. Let~$\tau$ be an automorphism of the disk satisfying $\tau(0)=e^{-i\theta}\xi$. Since $\dPC{0}{\xi}\leq\delta_0$, $d_{\adhDR{R}}(\phi_t\circ\tau,\phi_s\circ r_{\theta}\circ\tau)<\frac{\eps}{2}$. Moreover, one can find~$\zeta\in G$ such that $\dPC{\zeta}{\tau(0)}<A^{-1}e^{-R}$, and an automorphism~$\sigma$ of~$\set{D}$ such that $\sigma(0)=\zeta$ and $\dPCs{\adhDR{R}}{\sigma}{\tau}\leq c^{-1}\frac{\eps}{2}$, by statement~\eqref{lemDNSI<} of Lemma~\ref{lemDNS1Aut}. Putting all together, we obtain
    \[d_R(x,\phi_t(\zeta))\leq d_{\adhDR{R}}\left(\phi_s\circ r_{\theta}\circ\tau,\phi_t\circ\sigma\right)\leq d_{\adhDR{R}}\left(\phi_s\circ r_{\theta}\circ\tau,\phi_t\circ\tau\right)+d_{\adhDR{R}}\left(\phi_t\circ\tau,\phi_t\circ\sigma\right).\]
    In the right hand side, the first term is smaller than~$\frac{\eps}{2}$, and the second too, by statement~\eqref{dp<eps} of Lemma~\ref{lemh>dp>eps}. It follows that~$H$ is $(R,\eps)$-dense in~$X$. Hence,
    \[N'(X,R,\eps)\leq\card(H)=\card(F)\card(G)\leq Ce^{2R}N'\left(\set{T}_0,R+\delta_0,\frac{\eps}{2}\right).\]
      Taking $\frac{1}{R}\log$ and relevant limits, it yields $h(\fol)\leq2+h(\set{T}_0)$.

      For the other inequality, fix $\delta>0$, $F$ a $((1-\delta)R,2\eps)$-separated subset in~$\set{T}_0$ of maximal cardinality and~$G$ an $Ae^{-R}$-separated subset for the Poincaré distance in~$\DR{h/2}$ of maximal cardinality. Let us show that~$H=\{\phi_t(\zeta)~;~t\in F,~\zeta\in G\}$ is $(R,\eps)$-separated in~$X$. Take $t,s\in F$ and $\zeta,\xi\in G$ such that $d_R(\phi_t(\zeta),\phi_s(\xi))<\eps$. We need to prove that $t=s$, $\zeta=\xi$. Since $d_R(\phi_t(\zeta),\phi_s(\xi))<\eps$, one finds two automorphisms $\tau,\sigma\in\Aut(\set{D})$ such that $\tau(0)=\zeta$, $\sigma(0)=\xi$ and $d_{\adhDR{R}}(\phi_t\circ\tau,\phi_s\circ\sigma)<\eps$. By the same argument as in Lemma~\ref{lemhh'}, since $\dPC{\zeta}{\xi}\leq h$, it follows that $\dPCs{\adhDR{R}}{\tau}{\sigma}\leq c\eps$. Choosing well~$A$ above to apply Lemma~\ref{lemDNS1Aut}\eqref{lemDNSI>}, one gets that $\dPC{\zeta}{\xi}\leq Ae^{-R}$ and that $\zeta=\xi$, since they both belong to~$G$. Now, apply Lemma~\ref{lemhh'} to obtain that $d_{\adhDR{(1-\delta)R}}(\phi_t,\phi_s)<2\eps$. Since $t,s\in F$, one gets that $t=s$. Hence,~$H$ is $(R,\eps)$-separated in~$X$ and
      \[M(X,R,\eps)\geq\card(H)=\card(F)\card(G)\geq C^{-1}e^{2R}M(\set{T}_0,(1-\delta)R,2\eps).\]
      Taking $\frac{1}{R}\log$ and limits, it yields $h(\fol)\geq2+(1-\delta)h(\set{T}_0)$, and since~$\delta$ is arbitrary, this concludes the proof.
  \end{proof}

  \begin{rem}\label{remhh'} Lemma~\ref{lemhh'} will be very useful in the next subsection, since it will help us to get rid of reparametrization in~\eqref{defdR} to compute $h(\set{T}_0)$. Indeed, define another Bowen distance to be for $R>0$,
    \begin{equation}\label{defdR'}d_R'(t,s)=\sup_{\xi\in\adhDR{R}}\,d(\phi_t(\xi),\phi_s(\xi)),\qquad t,s\in\set{T}_0.\end{equation}
    Compared to~$d_R$, instead of being the distance between the two closest parametrizations of~$\leafu{t}$, $\leafu{s}$, it is the distance between the two canonical parametrizations~$\phi_t$ of~$\leafu{t}$ and~$\phi_s$ of~$\leafu{s}$. Once the Bowen distance $d_R'$ is given, we define completely analogous concepts of being $(R,\eps)'$-covering, dense or separated, for subsets of~$\set{T}_0$. Also, adding primes in the definition, we obtain an analogous notion of entropy denoted by~$h'(Y)$, for $Y\subset\set{T}_0$.

    In fact, $h(Y)=h'(Y)$. Indeed, it is clear that $d_R'\geq d_R$. Therefore, $h'(Y)\geq h(Y)$. On the other hand, Lemma~\ref{lemhh'} gives that for $\delta>0$, there is a $C>0$ such that $d_{(1-\delta)R}'\leq Cd_R$, for~$R$ sufficiently large. It follows that $h(Y)\geq(1-\delta)h'(Y)$ and since~$\delta$ is arbitrary, $h(Y)=h'(Y)$. It will be much more convenient to work with~$d_R'$ in the following to avoid reparametrization issues. Here, we showed that we still compute the same entropy.
  \end{rem}

  \subsection{Comparing with the weighted entropy}

  The results of the previous subsection are very general to non-singular laminations (see also~\cite[Proposition~4.1]{DNSII} for the singular case) and do not really depend on the suspension structure, with minor changes. What should be changed is that we would no longer consider the global transversal~$\set{T}_0$ but the union of local transversals in flow boxes. Also, there would be nothing like Lemma~\ref{lemhh'} since there would be no \emph{canonical} parametrization. In this subsection, we enter more deeply into the structure of a suspension, and into the underlying Fuchsian group.

  For $R>0$, denote by $\Gamma_R=\{\alpha\in\Gamma~;~\dPC{0}{\alpha(0)}\leq R\}$, and for $t,s\in\set{T}_0$, consider yet another Bowen-like distance
  \begin{equation}\label{defdGammaR}d_{\Gamma_R}(t,s)=\max\{d(\phi_t(\alpha(0)),\phi_s(\alpha(0)))~;~\alpha\in\Gamma_R\}.\end{equation}
  This distance is somehow equivalent to the distance~$d_R'$ defined in~\eqref{defdR'}, as is shown in the following.

  \begin{lem}\label{dGammad'} For $t,s\in\set{T}_0$ and $R>0$, we have $d_{\Gamma_R}(t,s)\leq d_R'(t,s)$. Moreover, for any~$\eps>0$, there exists~$\lambda>0$, independent on~$R$, $t$ and~$s$, such that $d_{\Gamma_{R+\delta_0}}(t,s)\leq\lambda$ implies $d_R'(t,s)\leq\eps$.
  \end{lem}

  \begin{proof} The first inequality is immediate. For the second, define~$\Delta_R=\cup_{\alpha\in\Gamma_R}\alpha(\adhDR{\delta_0})$. Then, $\adhDR{R}\subset\Delta_{R+\delta_0}$. Hence, for $t,s\in\set{T}_0$
    \begin{equation}\label{eqdR'dGammaR}\begin{aligned}d_R'(t,s)=d_{\adhDR{R}}(\phi_t,\phi_s)\leq d_{\Delta_{R+\delta_0}}(\phi_t,\phi_s)\leq&\max_{\alpha\in\Gamma_{R+\delta_0}}d_{\adhDR{\delta_0}}(\phi_t\circ\alpha,\phi_s\circ\alpha),\\
    \leq &\max_{\alpha\in\Gamma_{R+\delta_0}}d'_{\delta_0}(\rho(\alpha)(t),\rho(\alpha)(s)).\end{aligned}\end{equation}
    The function~$d'_{\delta_0}$ is uniformly continuous on the compact~$\set{T}_0\times\set{T}_0$. So, for~$\eps>0$, we can find $\lambda>0$ such that $d(t',s')\leq\lambda$ implies $d'_{\delta_0}(t',s')\leq\eps$, $t',s'\in\set{T}_0$. Suppose now that $d_{\Gamma_{R+\delta_0}}(t,s)\leq\lambda$. Then, for all $\alpha\in\Gamma_{R+\delta_0}$, $d(\rho(\alpha)(t),\rho(\alpha)(s))\leq\lambda$ and thus, $d'_{\delta_0}(\rho(\alpha)(t),\rho(\alpha)(s))\leq\eps$. It follows from~\eqref{eqdR'dGammaR} that $d_R'(t,s)\leq\eps$, as awaited. 
  \end{proof}

  The previous lemma will enable us to compare the entropy of~$\fol$ to some weighted entropy for~$\Gamma$ and~$\rho$, in the sense of Subsection~\ref{subsecweight}. Consider~$\mathcal{G}$ a set of generators such that $\mathcal{G}=\mathcal{G}^{-1}$, and for $\alpha\in\mathcal{G}$, define $\omega(\alpha)=\dPC{0}{\alpha(0)}$. Note that the weight~$\omega$ depends only on the Fuchsian group, and not on the representation~$\rho$.

  \begin{thm}\label{hThrho} With the notations above, there exists a constant $K'=K'(\Gamma)$, depending only on the Fuchsian group~$\Gamma$ and not on the representation~$\rho$, such that
    \[h(\rho,\mathcal{G},\omega)\leq h(\set{T}_0)\leq K'h(\rho,\mathcal{G},\omega).\]
  \end{thm}

  \begin{proof} Here, we use the results of Subsections~\ref{subseccover} and~\ref{subsecweight}. By~\eqref{rempropcovset} and Lemma~\ref{dGammad'}, $d_R'\geq d_{\Gamma_R}\geq d_R^{\rho,\mathcal{G},\omega}$. Hence, considering an $(R,\eps)'$-dense subset in~$\set{T}_0$, it is $\eps$-dense for $d_R^{\rho,\mathcal{G},\omega}$ and $h(\set{T}_0)=h'(\set{T}_0)\geq h(\rho,\mathcal{G},\omega)$, by Remark~\ref{remhh'}.

    For the other inequality, by Proposition~\ref{propcovest}, $d_{\Gamma_{(K^{-1}-\delta)N}}\leq d_N^{\rho,\mathcal{G},1}\leq d_{c_2N}^{\rho,\mathcal{G},\omega}$, for a constant~$K$ depending only on~$\Gamma$, any~$\delta\in\intoo{0}{K^{-1}}$ and~$N$ sufficiently large. Here, we used the notations of Lemma~\ref{lemlienhGLWh}. In particular, $c_2=\max\{\omega(\alpha)~;~\alpha\in\mathcal{G}\}$ only depends on~$\Gamma$. So, using Lemma~\ref{dGammad'}, if $d_{c_2N}^{\rho,\mathcal{G},\omega}(t,s)\leq\lambda$, then $d'_{(K^{-1}-\delta)N-\delta_0}(t,s)\leq\eps$. Considering a $\lambda$-dense subset for $d_{c_2N}^{\rho,\mathcal{G},\omega}$, it is $\left((K^{-1}-\delta)N-\delta_0,\eps\right)'$-dense. Taking $\frac{1}{N}\log$ and limits, we get $c_2h(\rho,\mathcal{G},\omega)\geq (K^{-1}-\delta)h'(\set{T}_0)$. Since~$\delta$ is arbitrary and by Remark~\ref{remhh'}, it yields $h(\set{T}_0)\leq c_2Kh(\rho,\mathcal{G},\omega)$, with $K'=c_2K$ only depending on~$\Gamma$.
  \end{proof}

  \subsection{Deriving some consequences}\label{subseccsq}

  We begin by constructing an example showing that the entropy is not a smooth invariant. Consider a complex manifold~$T$ and $f\colon T\to T$ to be an automorphism of positive entropy~$h(f)$. Take~$\rho\colon\Gamma\to\Aut(T)$ a representation such that $\rho(\alpha_1)=f$, and $\rho(\alpha_k)=\id$, $k\in\intent{2}{g}$, $\rho(\beta_k)=\id$, for $k\in\intent{1}{g}$. Here, we used the notations of the presentation of Proposition~\ref{presFuchs}. For such a representation, Ghys--Langevin--Walczak~\cite[Exemple~2.2]{GLW} showed that~$h_{\text{GLW}}\left(\rho(\Gamma),\rho(\mathcal{G})\right)=2h(f)$. Moreover, with the arguments of Lemma~\ref{lemlienhGLWh}, we also obtain that $h(\rho,\mathcal{G},\omega)=\frac{1}{\omega(\alpha_1)}h_{\text{GLW}}\left(\rho(\Gamma),\rho(\mathcal{G})\right)$. This has the following consequence.

  \begin{prop}\label{pasinvdiff} There exist two smoothly diffeomorphic foliations~$\fol_1$, $\fol_2$, such that $h(\fol_1)\neq h(\fol_2)$. In other words, the entropy in not a $\smooth$-invariant.
  \end{prop}

  \begin{proof} Denote by $\fol_1$, $\fol_2$ two suspensions obtained by the process above, with different but group-isomorphic Fuchsian groups~$\Gamma_1,\Gamma_2$, respective generators~$\mathcal{G}_1$, $\mathcal{G}_2$, and with same~$T$ and~$f$. Denote by~$\alpha_{1,j}\in\mathcal{G}_j$, $j\in\{1,2\}$, the generator of~$\Gamma_j$ sent to~$f$. Since the Riemann surfaces $\Sigma_j=\set{D}/\Gamma_j$ are diffeomorphic (being compact and of same genus) by a diffeomorphism sending~$\alpha_{1,1}$ to~$\alpha_{1,2}$, the foliations~$\fol_1$ and~$\fol_2$ are diffeomorphic. Moreover, by the paragraph above and by Theorems~\ref{hFhT} and~\ref{hThrho},
    \[\begin{aligned}h(\fol_1)&\leq2+K'(\Gamma_1)h(\rho_1,\mathcal{G}_1,\omega_1)=2+2h(f)\frac{K'(\Gamma_1)}{\omega_1(\alpha_{1,1})},\\
    h(\fol_2)&\geq2+h(\rho_2,\mathcal{G}_2,\omega_2)=2+\frac{2h(f)}{\omega_2(\alpha_{1,2})}.\end{aligned}\]

    So, fixing~$\Gamma_1$ and choosing~$\Gamma_2$ in consequence, it is enough to find it, along with a generator~$\alpha_{1,2}$ such that $\omega_2(\alpha_{1,2})=\dPC{0}{\alpha_{1,2}(0)}<\frac{\omega_1(\alpha_{1,1})}{K'(\Gamma_1)}$, which is a fixed constant. This can be achieved by Proposition~\ref{existdegFuchs} (see Figure~\ref{figpasinvdiff}). The whole point is that one can find a fixed generator~$\alpha_{1,1}\in\mathcal{G}_1$, which will be sent to the~$\alpha_{\eps}$ of Proposition~\ref{existdegFuchs} by the diffeomorphism. This is guaranteed by the construction of~$\alpha_{\eps}$ and by~\cite[10.4.2]{Farb}.
  \end{proof}

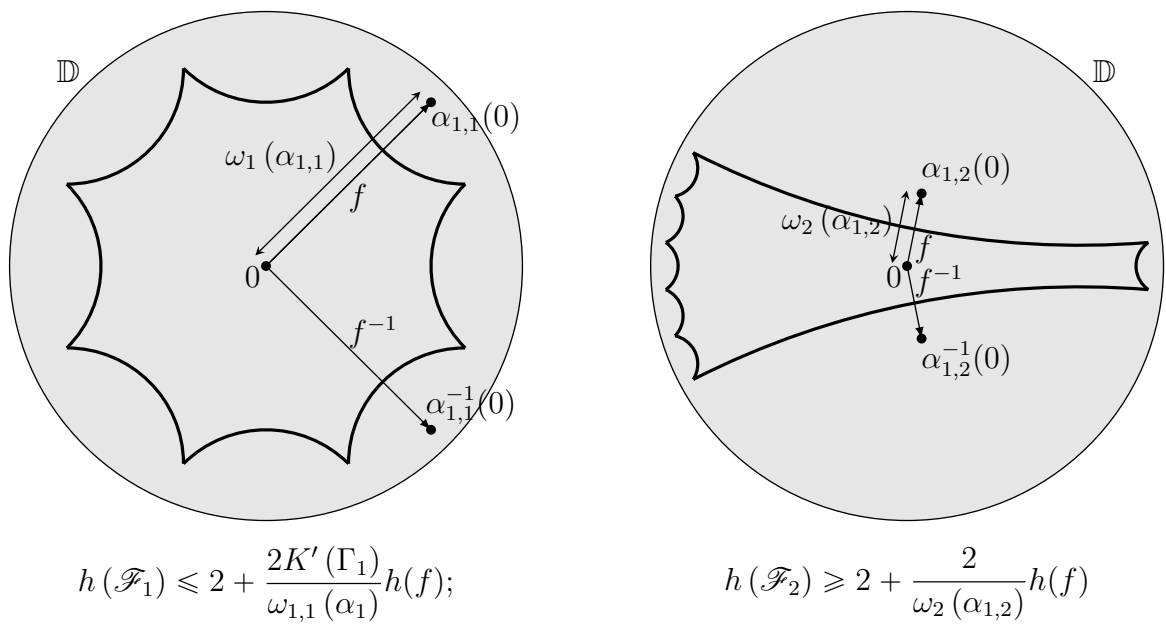
\begin{figure}[thb]
  \centering
  \input{hnondiffinv.tex}
  \caption{\label{figpasinvdiff} Fundamental domains in the proof of Proposition~\ref{existdegFuchs}.}
  \end{figure}


  \begin{rem} We believe that such a construction can show that there are suspensions for which all values in $\intoo{2}{\infty}$ can be achieved in their class of diffeomorphism. Indeed, the proof above shows that arbitrarily large entropy can be achieved. Moreover, Proposition~\ref{existdegFuchs} can be adapted to have a generator of arbitrary big weight. In Subsection~\ref{subsecZ}, we will even get an exact value of the entropy, of the form $2K_0h(f)$, with a constant~$K_0$ depending only on~$\Gamma$. However, this is not sufficient, since the constant~$K_0$ is not easily seen to depend continuously on the Fuchsian group, and we are not able to control it enough to show that arbitrarily small entropy greater than~$2$ can be achieved.
  \end{rem}

  So, the entropy is not a smooth invariant. However, in the case of suspensions and of homeomorphisms like the one in the proof of Proposition~\ref{pasinvdiff}, we have the following.

  \begin{prop}\label{h2inffaible} Let $\Phi\colon\Sigma_1\to\Sigma_2$ be a homeomorphism between two hyperbolic \mbox{Riemann} surfaces inducing an isomorphism $\varphi\colon\Gamma_1=\pi_1(\Sigma_1)\to\Gamma_2=\pi_1(\Sigma_2)$, $\Psi\colon T_1\to T_2$ a homeomorphism between compact metric spaces, and $\rho_i\colon\Gamma_i\to\Homeo(T_i)$, $i\in\{1,2\}$, be representations such that $\rho_2(\gamma_2)=\Psi\circ\rho_1\left(\varphi^{-1}\left(\gamma_2\right)\right)\circ\Psi^{-1}$, for $\gamma_2\in\Gamma_1$.

    Let~$\fol_i$ be the usual suspension of the representation~$\rho_i$, $i\in\{1,2\}$. Then, $h(\fol_1)=2$ if and only if $h(\fol_2)=2$ and $h(\fol_1)=+\infty$ if and only if $h(\fol_2)=+\infty$.
  \end{prop}

  \begin{proof} By Theorems~\ref{hFhT} and~\ref{hThrho}, we have $h(\fol_i)=2$ (resp. $h(\fol_i)=+\infty$) if and only if $h_{\text{GLW}}\left(\rho_i\left(\Gamma_i\right),\rho_i\left(\mathcal{G}_i\right)\right)=0$ (resp. $=+\infty$), for $i\in\{1,2\}$. Since~$\Psi$ is uniformly continuous, it is easy to obtain the statement of the proposition from this observation.
  \end{proof}


  
  It is worth noting that having infinite entropy is quite a wild phenomenon. Indeed, for a general compact lamination, it was shown by Dinh, Nguy\^{e}n and Sibony (see in~\cite{DNSI}, Theorem~3.10 and the few lines below Theorem~2.1) that if it is of class $\class{2+\alpha}$, for~$\alpha>0$, then it has finite entropy. In the case of suspensions, thanks to a result of Ghys, Langevin and Walczak, one has better.

  \begin{prop}\label{Lipimpfin} Let~$\fol$ be a Lipschitz suspension. Then $h(\fol)<+\infty$.
  \end{prop}

  \begin{proof} We use the notations of the previous subsection. If $\rho(\mathcal{G})$ contains only \mbox{Lipschitz} maps, it follows from~\cite[Proposition~2.7]{GLW} that $h_{\text{GLW}}\left(\rho(\Gamma),\rho(\mathcal{G})\right)$ is finite. We apply Theorem~\ref{hThrho} and Lemma~\ref{lemlienhGLWh} to conclude the proof.
  \end{proof}

  On the other hand, still by a previous work of Ghys, Langevin and Walczak, we have the following necessary condition for having minimal entropy~$2$. Note that a priori, having an invariant measure for the action of a non-amenable group~$\Gamma$ is a rare phenomenon (see for example~\cite{Furman,CantatDujardin,Roda} and references therein). 

  \begin{prop}\label{2impinv} Let $\fol$ be a suspension with entropy~$2$. Then, there exists a Borel probability measure~$\nu$ on~$T$ which is invariant by~$\Gamma$ (see Definition~\ref{defnmesstat} below).
  \end{prop}

  \begin{proof} It is shown in~\cite[Section~5]{GLW} that if~$\Gamma$ is a group of homeomorphisms of a compact metric space~$T$ and~$\mathcal{G}$ is a set of generators of~$\Gamma$ such that $\mathcal{G}=\mathcal{G}^{-1}$ and \mbox{$h_{\text{GLW}}(\Gamma,\mathcal{G})=0$}, then there exists a probability measure~$\nu$ on~$T$ which is invariant by~$\Gamma$. Since by Theorem~\ref{hThrho} and Lemma~\ref{lemlienhGLWh}, $h(\fol)\geq2+c_2h_{\text{GLW}}(\rho(\Gamma),\rho(\mathcal{G}))$, for $c_2=\max(\omega)$, this result applies.

  \end{proof}

\section{Measure-theoretic entropy}\label{secmesharm}

In this section, we will consider a measure-theoretic entropy for hyperbolic laminations, defined by Dinh, Nguy\^{e}n and Sibony~\cite{DNSI}.

\subsection{Harmonic, invariant and stationary measures}

\begin{defn}
  Roughly speaking, harmonic measures are quasi-invariant measures by the lamination~$\fol$. More precisely, let $\fol=\nspllam$ be a hyperbolic lamination and~$\mu$ be a Borel probability measure on~$X$. The measure~$\mu$ is said to be \emph{harmonic} if for every flow box $U\simeq\set{D}\times\set{T}$, there exists a Radon measure~$\nu$ on~$\set{T}$ and harmonic functions $h_t\colon\set{D}\to\set{R}_+$, for $t\in\set{T}$, such that for every integrable function $f\colon U\to\set{R}$,
  \begin{equation} \label{defmesharm}\int fd\mu=\int_{\set{T}}\int_{\set{D}}h_t(z)f(t,z)\metPC{}(z)d\nu(t).\end{equation}
  If moreover the harmonic functions~$h_t$, $t\in\set{T}$, are all constant,~$\mu$ is called an invariant measure.
\end{defn}

There are analogous concepts for group actions.

\begin{defn}\label{defnmesstat}
  Let $\Gamma$ be a group acting on a topological space~$T$ and~$m$ be a probability measure on~$\Gamma$. A Borel measure~$\nu$ on~$T$ is called \emph{invariant} by~$\Gamma$ if for every~$\gamma\in\Gamma$, $\gamma_*\nu=\nu$. It is called \emph{$m$-stationary} if it is invariant on average, \emph{i.e.} if
  \[\int_{\Gamma}\gamma_*\nu dm(\gamma)=\nu.\]
\end{defn}

In fact, in the case of suspensions, there is a bridge between these two couples of notions with the following result.

\begin{thm}[Alvarez~\cite{Alvarez}]\label{thmAlvarez} Let $\Sigma=\set{D}/\Gamma$ be a smooth compact Riemann surface of genus~$g\geq2$. There exists a probability measure~$m$ on~$\Gamma$ such that $m(\gamma)>0$ for every~$\gamma\in\Gamma$ and satisfying the following. Let~$\fol$ be a suspension obtained as in Subsection~\ref{subsecsusp}. Then, the map
  \[\mu\mapsto h_t(0)d\nu(t),\]
  with the notations of~\eqref{defmesharm} on the flow box $U\times\set{T}_0$, for~$U$ a small neighbourhood of~$0$ in~$\set{D}$, sends bijectively harmonic measures for~$\fol$ on $m$-stationary measures for~$\Gamma$. Moreover, it maps bijectively invariant measures for~$\fol$ on invariant measures for~$\Gamma$, and bijectively ergodic measures for~$\fol$ on ergodic measures for~$\Gamma$.
\end{thm}

The theorem above is valid in a broader context, but we state it in the one that is relevant here. Note that Proposition~\ref{2impinv} now gives the following.

 \begin{prop}\label{2impinvv2} Let $\fol$ be a suspension with entropy~$2$. Then, there exists an invariant measure for~$\fol$.
 \end{prop}

\subsection{Entropy of a harmonic measure}

For the dynamics of a single map, the usual definition of the entropy of a measure relies on measurable partitions. In our context, such a definition does not work, but thanks to Katok~\cite{Katok} and Brin--Katok~\cite{BrinKatok} theorems, one can still have analogies with only a notion of Bowen balls.

\begin{defn}[See~{\cite[Section~4]{DNSI}}]\label{defentmes} Let~$\fol=\nspllam{}$ be a hyperbolic lamination and~$\mu$ a harmonic measure for~$\fol$. For $R,\eps,\delta>0$, denote by $N(R,\eps,\delta)$ the minimal amount of $(R,\eps)$-Bowen balls required to cover a subset of $\mu$-measure at least $1-\delta$. Define the \emph{entropy of~$\mu$} to be
  \[h(\mu)=\sup_{\delta>0}\sup_{\eps>0}\limsup_{R\to+\infty}\frac{1}{R}\log N(R,\eps,\delta).\]
  Given $x\in X$, one also defines \emph{local upper and lower entropies of~$\mu$ at~$x$}
  \[\begin{aligned} h^+(\mu,x)&=\sup_{\eps>0}\limsup_{R\to+\infty}-\frac{1}{R}\log\mu(B_R(x,\eps)),\\
  h^-(\mu,x)&=\sup_{\eps>0}\liminf_{R\to+\infty}-\frac{1}{R}\log\mu(B_R(x,\eps)).\end{aligned}\]
  Finally, consider a local flow box $U\simeq\set{D}\times\set{T}$ and a local disintegration of~$\mu$ into a leafwise harmonic density and a transversal Radon measure~$\nu$ as in~\eqref{defmesharm}. Rescale the~$h_t$ and~$\nu$ if necessary to have that $h_t(0)=1$, for all $t\in\set{T}$. Let~$\pi_{\set{T}}\colon U\to\set{T}$ be the projection on the second coordinate. Define the \emph{local transversal upper and lower entropies of~$\mu$ at~$x$} to be
  \[\begin{aligned} \wt{h}^+(\mu,x)&=\sup_{\eps>0}\limsup_{R\to+\infty}-\frac{1}{R}\log\nu(\pi_{\set{T}}(B_R(x,\eps))),\\
  \wt{h}^-(\mu,x)&=\sup_{\eps>0}\liminf_{R\to+\infty}-\frac{1}{R}\log\nu(\pi_{\set{T}}(B_R(x,\eps))).\end{aligned}\]
\end{defn}

In our context, such quantities were studied by Dinh, Nguy\^{e}n and Sibony~\cite{DNSI} (see also~\cite{Bac4} for a singular context). In particular, they show the following.

\begin{thm}[Dinh--Nguy\^{e}n--Sibony~{\cite[Theorem~4.2 and Proposition~4.5]{DNSI}}]\label{ordreent} Consider $\fol=\nspllam$ a compact hyperbolic lamination and~$\mu$ a harmonic measure. Then, the quantities $x\mapsto\wt{h}^{\pm}(\mu,x)$ and $x\mapsto h^{\pm}(\mu,x)$ are leafwise constant with $h^{\pm}(\mu,x)=2+\wt{h}^{\pm}(\mu,x)$. In particular, if~$\mu$ is ergodic, then all these quantities are constant almost everywhere. Denote them by $\wt{h}^{\pm}(\mu)$, $h^{\pm}(\mu)$. We have
  \[2\leq h^-(\mu)\leq h(\mu)\leq h^+(\mu)\leq h(\fol).\]
\end{thm}

In~\cite[Problems~5 and~6]{DNSI}, the three authors ask whether $h^+(\mu)=h^-(\mu)$ (a Brin--Katok-type theorem) and whether one has the variational principle $h(\fol)=\sup_{\mu}h(\mu)$. Except when the entropy equals~$2$, for which it is trivial, there are no examples yet in which one can show it holds or it does not. Below, in Theorem~\ref{casZ}, we will study a toy-example in which one can use the classical theorems for maps to show them on foliations.

Before ending this subsection, let us give a more adapted way to compute the transversal entropies in the case of suspensions.

\begin{lem}\label{enttransmes} Let $\fol$ be a suspension with the usual notations, $\mu$ be a harmonic measure for~$\fol$ and recall the Bowen distances $d_R'$ defined in~\eqref{defdR'}. For $t\in T$, define
  \[B_R^{\set{T}}(t,\eps)=\{s\in T~;~d_R'(s,t)<\eps\},\qquad B_R^{\set{T},\Gamma}(t,\eps)=\{s\in T~;~d_{\Gamma_R}(t,s)<\eps\},\]
  where $d_{\Gamma_R}$ is defined in~\eqref{defdGammaR}. Then, if~$\nu$ denotes the measure on~$T$ given by Theorem~\ref{thmAlvarez},
  \[\begin{aligned}\wt{h}^+(\mu,t)&=\sup_{\eps>0}\limsup_{R\to+\infty}-\frac{1}{R}\log\nu\left(B_R^{\set{T}}(t,\eps)\right)=\sup_{\eps>0}\limsup_{R\to+\infty}-\frac{1}{R}\log\nu\left(B_R^{\set{T},\Gamma}(t,\eps)\right),\\
  \wt{h}^-(\mu,t)&=\sup_{\eps>0}\liminf_{R\to+\infty}-\frac{1}{R}\log\nu\left(B_R^{\set{T}}(t,\eps)\right)=\sup_{\eps>0}\liminf_{R\to+\infty}-\frac{1}{R}\log\nu\left(B_R^{\set{T},\Gamma}(t,\eps)\right).\end{aligned}\]
\end{lem}

\begin{proof} The right equalities are consequences of Lemma~\ref{dGammad'}. For the left ones, with the notations of Definition~\ref{defentmes}, for some small~$h>0$ and if~$\eps$ is sufficiently small,
  \[\pi_{\set{T}}(B_R(t,\eps))=\{s\in T~;~\exists\zeta\in\DR{h}~\text{with}~d_R(\phi_s(\zeta),t)<\eps\}.\]
  Since $d_R'>d_R$, it is clear that we have the inequality ``$\leq$'' in the two equalities to be shown. Now, consider $s\in\pi_{\set{T}}(B_R(t,\eps))$ and take $\zeta\in\DR{h}$ with $d_R(\phi_s(\zeta),t)<\eps$. By the same arguments as in Lemma~\ref{lemhh'}, we have $\dPC{0}{\zeta}<Ae^{-R}$. It follows by statement~\eqref{lemDNSI<} of Lemma~\ref{lemDNS1Aut} that $d_{R(1-\delta)}(\phi_s(\zeta),\phi_s(0))<\eps$, for any $\delta>0$ and~$R$ sufficiently large. Hence, $d_{R(1-\delta)}(t,s)<2\eps$ and $d'_{R(1-\delta)^2}(s,t)<4\eps$ by Lemma~\ref{lemhh'}. We conclude by the same arguments as in Remark~\ref{remhh'}.
\end{proof}

\subsection{Case of a single map} \label{subsecZ}

Here, we consider a suspension~$\fol$, for which $\rho(\Gamma)=\langle f\rangle\simeq\set{Z}$ is generated by a single map~$f$, and with the usual notations. More precisely, for $\alpha\in\mathcal{G}$, denote by $n(\alpha)\in\set{Z}$ the integer such that $\rho(\alpha)=f^{n(\alpha)}$. Our main result is the following.

\begin{thm}\label{casZ} In the setup above, there exists a constant $K_0$ that does not depend on~$f$, such that
  \begin{enumerate}[label=(\roman*),ref=\roman*]
  \item\label{itemhtop} $h(\fol)=2+2K_0h_{\text{top}}(f)$
  \item\label{itemBK} Every harmonic measure~$\mu$ on~$\fol$ is actually invariant. Denote by~$\nu$ its trace on~$T$ given by Theorem~\ref{thmAlvarez}. Then, if~$\mu$ is ergodic,
    \[h^-(\mu)=h(\mu)=h^+(\mu)=2+2K_0h_{\nu}(f).\]
  \item\label{itemvar} $h(\fol)=\sup_{\mu}h(\mu)=\sup_{\mu~\text{ergodic}}h(\mu)$.
  \end{enumerate}
\end{thm}

Here, the fact that~$K_0$ does not depend on~$f$ means more precisely that it depends only on the Fuchsian group~$\Gamma$, and on the powers~$n(\alpha)$, for $\alpha\in\mathcal{G}$, but not on other choices. Moreover, we denoted by $h_{\text{top}}(f)$ the topological entropy of~$f$ in the usual sense, and by $h_{\nu}(f)$ its measure-theoretic entropy for~$\nu$. Note that item~\eqref{itemvar} is a consequence of the first two, and of the classical variational principle for maps (see~\cite[Theorem~8.6 and Corollary~8.6.1]{Wal}). Let us emphasize on the~$2$ in the theorem above, which comes from the fact that powers of~$f$ from~$-n$ to~$n$ are considered, instead of classically from~$0$ to~$n$.

\begin{rem} Theorem~\ref{casZ} shows that there is no converse to Proposition~\ref{2impinv}. Even a weak version stating that~$\mu$ invariant implies $h(\mu)=0$ does not hold. Here, even if all harmonic measures are invariant, we don't necessarily have that $h(\fol)=2$. Other examples of random dynamical systems for which all stationary measures are invariant (so-called \emph{stiffness}) can be found in~\cite{CantatDujardin,Roda} and references therein.
\end{rem}

Let us start by finding the constant~$K_0$. For~$R>0$, denote by
\[n(R)=\max\{n\in\set{N}~;~\exists\alpha\in\Gamma,~\dPC{0}{\alpha(0)}\leq R~\text{and}~\rho(\alpha)=f^n\}.\]

\begin{lem}\label{existK0} The limit $K_0=\lim_{R\to+\infty}\frac{n(R)}{R}$ exists in~$\set{R}_+^*$.
\end{lem}

\begin{proof} It is easy to see that $R\mapsto-n(R)$ is subadditive. Hence, the limit of $\frac{n(R)}{R}$ exists in $\intcc{0}{\infty}$ by~\cite[Theorem~4.9]{Wal}. Moreover, there exists~$K_1>1$ with $K_1^{-1}R\leq n(R)\leq K_1R$, for sufficiently large~$R$. Indeed, on the one hand $n(R)\geq\ientp{\frac{R}{\dPC{0}{\alpha(0)}}}\Cmod{n(\alpha)}$, for $\alpha\in\mathcal{G}$. On the other hand, using the notations of Proposition~\ref{propcovest} and applying it to $\delta=K^{-1}/2$, one gets $n(R)\leq \sentp{2KR}\max_{\alpha\in\mathcal{G}}n(\alpha)$. Therefore, $K_0$ is neither~$0$ nor~$\infty$. Recall that for $x\in\set{R}$, we denote by $\ient{x}$ (resp. $\sent{x}$) the highest (resp. lowest) integer $n\in\set{Z}$ such that $n\leq x$ (resp $n\geq x$). 
\end{proof}

\begin{proof}[Proof of item~\eqref{itemhtop} of Theorem~\ref{casZ}] Note that the constant~$K_0$ in the lemma above depends only on the Fuchsian group and on the powers $n(\alpha)$, $\alpha\in\mathcal{G}$. Now, recall the Bowen distance~$d_{\Gamma_R}$ defined in~\eqref{defdGammaR} and denote by $d_n^{f^{\pm1}}(t,s)=\max_{i\in\intent{-n}{n}}d(f^i(s),f^i(s))$, for $s,t\in T$ and~$n\in\set{N}$. Finally, denote by $d_n^f(s,t)=\max_{i\in\intent{0}{n}}d(f^i(s),f^i(t))$ the classical Bowen distance. Write $B_{\Gamma_R}(t,\eps)$, $B_n^{f^{\pm1}}(t,\eps)$, $B_n^f(t,\eps)$ for the corresponding Bowen balls and $N_{\Gamma}(T,R,\eps)$, $N_f(T,n,\eps)$, $N_{f^{\pm1}}(T,n,\eps)$ for the corresponding minimal number of Bowen balls required to cover~$T$, as in Definition~\ref{defnN(Y,R,E)}. Then, the Bowen definition of the topological entropy~\cite[Definition~7.9 and~7.10]{Wal} is
  \[h_{\text{top}}(f)=\sup_{\eps>0}\limsup_{n\to+\infty}\frac{1}{n}\log N_f(T,n,\eps).\]
  On the other hand, for $t\in T$ and $n,\eps>0$, $B_n^{f^{\pm1}}(t,\eps)=f^n\left(B_{2n}^f\left(f^{-n}(t),\eps\right)\right)$, so that $N_{f^{\pm1}}(T,n,\eps)=N_f(T,2n,\eps)$. Now, if we define
  \[I(R)=\{i\in\set{Z}~;~\exists\alpha\in\Gamma,~\dPC{0}{\alpha(0)}\leq R,~\rho(\alpha)=f^i\},\]
  by definition $I(R)\subset\intent{-n(R)}{n(R)}$. Moreover, the maximal step between two consecutive points of~$I(R)$ is at most~$P=\max_{\alpha\in\mathcal{G}}(n(\alpha))$. Therefore, by a uniform continuity argument on~$f^p$, for $p\in\intent{1}{P}$, as in Lemma~\ref{dGammad'}, we have that for all~$\eps>0$, there exists~$\lambda>0$ such that $B_{\Gamma_R}(t,\lambda)\subset B_{n(R)}^{f^{\pm1}}(t,\eps)\subset B_{\Gamma_R}(t,\eps)$, for any~$t\in T$. Hence,
  \[N_{\Gamma}(\set{T}_0,R,\lambda)\geq N_{f^{\pm1}}(T,n(R),\eps)=N_f(T,2n(R),\eps)\geq N_{\Gamma}(\set{T}_0,R,\eps).\]
  We conclude with Lemma~\ref{existK0}, usual arguments that~$d_{\Gamma_R}$ and $d_R$ define the same entropy, and Theorem~\ref{hFhT}.
\end{proof}

Note that by the above inclusions, and by Lemma~\ref{enttransmes}, we have the following.

\begin{lem}\label{enttransmes2} With the setup of Theorem~\ref{casZ}, for $t\in T$,
  \[\begin{aligned} \wt{h}^+(\mu,t)&=K_0\sup_{\eps>0}\limsup_{n\to+\infty}-\frac{1}{n}\log\nu\left(B_n^{f^{\pm1}}(t,\eps)\right),\\
  \wt{h}^-(\mu,t)&=K_0\sup_{\eps>0}\liminf_{n\to+\infty}-\frac{1}{n}\log\nu\left(B_n^{f^{\pm1}}(t,\eps)\right).\end{aligned}\]
\end{lem}

By Martingale theorem and Hewitt--Savage $0$--$1$ law~\cite{Blackwell}, one can show that every bounded $\rho_*m$-harmonic function on~$\set{Z}$ is constant. This is actually a very simple case of the Choquet--Deny property of virtually nilpotent groups~\cite{FHTF}. In particular, it implies that any $m$-stationary measure~$\nu$ is actually invariant by~$f$. Hence, by Theorems~\ref{thmAlvarez} and~\ref{ordreent}, the proof of Theorem~\ref{casZ} is reduced to the following.

\begin{prop} Let $f\colon T\to T$ be a homeomorphism of a compact metric space and~$\nu$ be an ergodic invariant measure. Then, for $\nu$-almost every $t\in T$,
  \[\sup_{\eps>0}\liminf_{n\to+\infty}-\frac{1}{n}\log\nu\left(B_n^{f^{\pm1}}(t,\eps)\right)=\sup_{\eps>0}\limsup_{n\to+\infty}-\frac{1}{n}\log\nu\left(B_n^{f^{\pm1}}(t,\eps)\right)=2h_{\nu}(f).\]
\end{prop}

This is almost exactly like Brin--Katok theorem~\cite{BrinKatok}, except that two-sided Bowen balls $B_n^{f^{\pm1}}$ are involved, instead of usual Bowen balls~$B_n^f$. Of course, our arguments will use Brin--Katok's theorem and its proof.

\begin{proof} We only do the proof for $h_{\nu}(f)<+\infty$, since it is easy to adapt it for infinite entropy. Let us start by showing that the $\liminf$ is greater than $2h_{\nu}(f)$. This is the most complicated part of the proof of Brin--Katok, but they show~\cite[p.~37]{BrinKatok} that for every~$\delta>0$, there exists~$\eps>0$ with
  \[\left(\nu\left\{t\in T~;~\nu\left(B_n^f(t,\eps)\right)>\exp\left(-(h_{\nu}(f)-\delta)n\right)\right\}\right)_{n\in\set{N}}\]
  summable. Since $\nu\left(B_n^{f^{\pm1}}(t,\eps)\right)=\nu\left(B_{2n}^f(f^{-n}(t),\eps)\right)$ because~$\nu$ is invariant,
  \[\left(\nu\left\{t\in T~;~\nu\left(B_n^{f^{\pm1}}(t,\eps)\right)>\exp\left(-2(h_{\nu}(f)-\delta)n\right)\right\}\right)_{n\in\set{N}}\]
  is also summable. Here, we used again the $\nu$-invariance by~$f$. Just as Brin and Katok, we conclude by Borel--Cantelli theorem that for $\nu$-almost every $t\in T$,
  \begin{equation}\label{h->}\sup_{\eps>0}\liminf_{n\to+\infty}-\frac{1}{n}\log\nu\left(B_n^{f^{\pm1}}(t,\eps)\right)\geq2h_{\nu}(f).\end{equation}

  Now, let us proceed with showing that the $\limsup$ is lower than $2h_{\nu}(f)$. Just as Brin and Katok, we use Shannon--McMillan--Breiman theorem, except that an adapted version is needed. Consider~$\xi$ a finite measurable partition of~$T$. For $x\in T$, denote by $c_n^{\pm}(x)$ the element of the partition $\bigvee_{i=-n}^nf^{-i}(\xi)$ containing~$x$. Using Birkhoff theorem for both~$f$ and~$f^{-1}$, one can adapt the proof of~\cite[Theorem~2.5]{Parry} to show that
  \[-\frac{1}{2n}\log\nu\left(c_n^{\pm}(t)\right)\underset{n\to+\infty}{\longrightarrow}h_{\nu}(f,\xi)\qquad\text{almost everywhere},\]
  where $h_{\nu}(f,\xi)\leq h_{\nu}(f)$ is the entropy of~$f$, with respect to the partition~$\xi$. Consider a finite measurable partition~$\xi$, the elements of which all have diameter at most~$\eps$. Then, $c_n^{\pm}(t)\subset B_n^{f^{\pm1}}(t,\eps)$. Hence,
  \[\limsup_{n\to+\infty}-\frac{1}{n}\log\nu\left(B_n^{f^{\pm1}}(t,\eps)\right)\leq2\limsup_{n\to+\infty}-\frac{1}{2n}\log\nu\left(c_n^{\pm}(x)\right)\leq 2h_{\nu}(f),\]
  almost everywhere. Taking the $\sup$ with $\eps>0$, and together with~\eqref{h->}, this finishes the proof.
\end{proof}


     
\bibliography{Suspensions}

\bibliographystyle{plain}

\end{document}

%% file: Deg-Fuchs.tex
\def \globalscale{10}
\begin{tikzpicture}[y=0.80pt, x=0.80pt, yscale=\globalscale, xscale=\globalscale, inner sep=0pt, outer sep=0pt,line width=0.5pt,draw=black]
\draw[white] (10.6066,10.6066) node[above right,black,xshift=2,yshift=2] {$\set{D}$} circle (0.1);
\draw[white] (0,0) node[right,black,xshift=25] {$\mu$} circle (0.1);
\draw[white] (0,0) node[left,black,xshift=-25] {$\mu$} circle (0.1);
\draw[white] (0,0) node[above left,black,xshift=-21,yshift=10] {$\mu$} circle (0.1);
\draw[white] (0,0) node[below left,black,xshift=-21,yshift=-11] {$\mu$} circle (0.1);
\draw[white] (0,0) node[above left,black,xshift=-11,yshift=21] {$\mu$} circle (0.1);
\draw[white] (0,0) node[below left,black,xshift=-11,yshift=-22] {$\mu$} circle (0.1);
\draw[white] (0,0) node[above,black,yshift=18,xshift=14] {$\nu$} circle (0.1);
\draw[white] (0,0) node[below,black,yshift=-18,xshift=14] {$\nu$} circle (0.1);
\filldraw[black,fill opacity=0.1] (0,0) circle (15);
\draw (2.8978,-0.7764) arc (-15:15:3);
\draw (-0.7764,2.8978) arc (105:255:3);
\draw (2.4148,0.6470) arc (15:105:2.5);
\draw (-0.6470,-2.4148) arc (-105:-15:2.5);
\draw (-9.659,-2.588) -- (14.4885,3.882);
\draw (-9.659,2.588) -- (9.659,-2.588);
\draw (-7.071,-7.071) -- (0,0);
\draw (-7.071,7.071) -- (0,0);
\draw (-2.588,-9.659) -- (0,0);
\draw (-3.882,14.4885) -- (0,0);
\draw (-3.882,14.4885) -- (14.4885,3.882);
\draw (5.30325,9.18525) -- (0,0);
\filldraw[black] (4.8953,8.479) node[left,xshift=-2] {$\alpha_0(0)$} circle (0.2);
\filldraw[black] (0,0) node[below right,yshift=-2] {$0$} circle (0.2);
\filldraw[black] (9.659,2.588) node[above left,xshift=-1.5,yshift=2.5] {$A_1$} circle (0.2);
\filldraw[black] (9.659,-2.588) node[below left,xshift=-1.5,yshift=-3] {$A_8$} circle (0.2);
\filldraw[black] (-9.659,2.588) node[below left,xshift=-1.5] {$A_4$} circle (0.2);
\filldraw[black] (-9.659,-2.588) node[above left,xshift=-1.5] {$A_5$} circle (0.2);
\filldraw[black] (-7.071,-7.071) node[above left,xshift=-2] {$A_6$} circle (0.2);
\filldraw[black] (-7.071,7.071) node[below left,xshift=-3] {$A_3$} circle (0.2);
\filldraw[black] (-2.588,9.659) node[above left,xshift=-1.5] {$A_2$} circle (0.2);
\filldraw[black] (-2.588,-9.659) node[below left,xshift=-1.5] {$A_7$} circle (0.2);
\filldraw[black] (14.4885,3.882) node[above right,xshift=1.5] {$e^{i\mu/2}$} circle (0.15);
\filldraw[black] (-3.882,14.4885) node[above left,xshift=-1.5] {$e^{i(\nu+\mu/2)}$} circle (0.15);
\draw[line width=1.2pt] (9.659,2.588) arc (160.1361:199.8639:7.3576);
\draw[line width=1.2pt] (-9.659,-2.588) arc (-19.8639:19.8639:7.3576);
\draw[line width=1.2pt] (-9.659,2.588) arc (-49.8639:-10.1361:7.3576);
\draw[line width=1.2pt] (-7.071,-7.071) arc (10.1361:49.8639:7.3576);
\draw[line width=1.2pt] (-7.071,7.071) arc (-79.8639:-40.1361:7.3576);
\draw[line width=1.2pt] (-2.588,-9.659) arc (40.1361:79.8639:7.3576);
\draw[line width=1.2pt] (-2.588,9.659) arc (216.0375:263.9625:17.4105);
\draw[line width=1.2pt] (9.659,-2.588) arc (-263.9625:-216.0375:17.4105);
\filldraw[black] (2.7853,4.8141) node[above right,xshift=3] {$B_2$} circle (0.2);
\filldraw[black] (2.7853,-4.8141) node[below right,xshift=3] {$B_1$} circle (0.2);
\end{tikzpicture}

%% file: hnondiffinv.tex
\def \globalscale{10}
\begin{tikzpicture}[y=0.80pt, x=0.80pt, yscale=\globalscale, xscale=\globalscale, inner sep=0pt, outer sep=0pt,line width=0.5pt,draw=black]
\draw[white] (-8.7,8.7) node[above left,black] {$\set{D}$} circle (.1);
\filldraw[black,fill opacity=0.1] (0,0) circle (12);
\draw[line width=1.2pt] (9.32265,3.86166) arc (135:225:5.46108);
\draw[line width=1.2pt] (3.86166,9.32265) arc (180:270:5.46108);
\draw[line width=1.2pt] (-3.86166,9.32265) arc (225:315:5.46108);
\draw[line width=1.2pt] (-9.32265,3.86166) arc (270:360:5.46108);
\draw[line width=1.2pt] (-9.32265,-3.86166) arc (315:405:5.46108);
\draw[line width=1.2pt] (-3.86166,-9.32265) arc (0:90:5.46108);
\draw[line width=1.2pt] (3.86166,-9.32265) arc (45:135:5.46108);
\draw[line width=1.2pt] (9.32265,-3.86166) arc (90:180:5.46108);
\filldraw[black] (0,0) node[below left,xshift=-2] {$0$} circle (0.2);
\filldraw[black] (7.72313,7.72313) node[below right] {$\alpha_{1,1}(0)$} circle (0.2);
\filldraw[black] (7.72313,-7.72313) node[above right,yshift=2,xshift=-2] {$\alpha_{1,1}^{-1}(0)$} circle (0.2);
\draw[->,>=latex](0,0) -- (7.72313,7.72313) node[midway,below right] {$f$};
\draw[<->,>=stealth](-0.5,0.5) -- (7.22313,8.22313) node[midway,above left] {$\omega_1\left(\alpha_{1,1}\right)$};
\draw(0,0) -- (7.72313,7.72313) ;
\draw[->,>=latex](0,0) -- (7.72313,-7.72313) node[midway,above right] {$f^{-1}$};
\draw[white] (0,-15) node[black] {$h\left(\fol_1\right)\leq2+\dfrac{2K'\left(\Gamma_1\right)}{\omega_{1,1}\left(\alpha_1\right)}h(f);$} circle (.1);

\draw[white] (38.7,8.7) node[above right,black] {$\set{D}$} circle (.1);
\filldraw[fill=black, fill opacity=0.1] (30,0) circle (12);
\draw[line width=1.2pt] (41.2898,1.1119) arc (125.2999:234.7001:1.36245);
\draw[line width=1.2pt] (18.7101,-1.1119) arc (-54.7001:54.7001:1.36245);
\draw[line width=1.2pt] (18.7101,1.1119) arc (-65.9501:43.4501:1.36245);
\draw[line width=1.2pt] (19.14406,3.2931) arc (-77.2001:32.2001:1.36245);
\draw[line width=1.2pt] (19.14406,-3.2931) arc (-43.4501:65.9501:1.36245);
\draw[line width=1.2pt] (19.9951,-5.34477) arc (-32.2001:77.2001:1.36245);
\draw[line width=1.2pt] (41.2898,-1.1119) arc (85.3503:117.1497:39.6269);
\draw[line width=1.2pt] (19.9951,5.34477) arc (-117.1497:-85.3503:39.6269);
\filldraw[black] (30,0) node[below left,xshift=-2] {$0$} circle (0.2);
\draw[white] (30,-15) node[black] {$h\left(\fol_2\right)\geq2+\dfrac{2}{\omega_2\left(\alpha_{1,2}\right)}h(f)$} circle (.1);
\filldraw[black] (30.6807,3.4221) node[above right,yshift=3] {$\alpha_{1,2}(0)$} circle (0.2);
\filldraw[black] (30.6807,-3.4221) node[below right,yshift=-1] {$\alpha_{1,2}^{-1}(0)$} circle (0.2);
\draw[<->,>=stealth] (29.3065,0.13795) -- (29.9872,3.56005) node[midway,left,xshift=-2,yshift=2] {$\omega_2\left(\alpha_{1,2}\right)$};
\draw[<-,>=latex] (30.6807,3.4221) -- (30,0) node[midway,below right,yshift=-2] {$f$};
\draw[->,>=latex] (30,0) -- (30.6807,-3.4221) node[midway,above right,yshift=2] {$f^{-1}$};
\end{tikzpicture}